\documentclass[12pt,a4paper]{article}
\usepackage[dvipdfmx]{graphicx,color}
\usepackage{amsmath,amsfonts,amssymb}
\usepackage{hyperref}
\usepackage{ulem}

\usepackage[numbers]{natbib}
\newcommand{\argsup}{\mathop{\rm arg~sup}\limits}

\newcommand{\rs}[1]{{\mbox{\scriptsize \sc #1}}}
\newcommand{\vc}[1]{{\boldsymbol #1}} 
\newcommand{\vcn}[1]{{\bf #1}}

\newcommand{\sr}[1]{{\cal #1}}
\newcommand{\dd}[1]{\mathbb{#1}}

\newcommand{\rmn}[1]{\if#11I\else {\if#12I\hspace{-0.12ex}I\hspace{-0.85ex}\else {\if #13I\hspace{-0.16ex}I\hspace{-0.16ex}I\hspace{-1.6ex}\else I\hspace{-1.2ex}V \fi} \fi} \fi}

\newcommand{\eq}[1]{(\ref{eq:#1})}
\newcommand{\lem}[1]{Lemma~\ref{lem:#1}}
\newcommand{\cor}[1]{Corollary~\ref{cor:#1}}
\newcommand{\thr}[1]{Theorem~\ref{thr:#1}}

\newcommand{\fig}[1]{Figure~\ref{fig:#1}}

\newcommand{\sectn}[1]{Section~\ref{sect:#1}}

\newcommand{\lemt}[1]{\ref{lem:#1}}

\newcommand{\thrt}[1]{\ref{thr:#1}}

\newcommand{\sect}[1]{\ref{sect:#1}}

\newcommand{\br}[1]{\langle #1 \rangle}

\newcommand{\ol}{\overline}
\newcommand{\ul}{\underline}
\newcommand{\pend}{\hfill \thicklines \framebox(6.6,6.6)[l]{}}

\newenvironment{proof}{\noindent {\sc  Proof.} \rm}{\pend}
\newenvironment{proof*}[1]{\noindent {\sc  #1} \rm}{\pend \medskip}

\newtheorem{theorem}{Theorem}[section]
\newtheorem{lemma}{Lemma}[section]

\newtheorem{remark}{Remark}[section]

\newtheorem{corollary}{Corollary}[section]

\newcommand{\setsection}[2] {
\setcounter{section}{#1}
\setcounter{subsection}{0}
\setcounter{equation}{0}
\setcounter{conjecture}{0}
\setcounter{assumption}{0}
\setcounter{question}{0}
\setcounter{definition}{0}
\setcounter{theorem}{0}
\setcounter{corollary}{0}
\setcounter{lemma}{0}
\setcounter{proposition}{0}
\setcounter{remark}{0}
\setcounter{appen}{0}
\setsection*{\large \bf \thesection. #2}}

\newenvironment{mylist}[1]{\begin{list}{}
{\setlength{\itemindent}{#1mm}}
{\setlength{\itemsep}{0ex plus 0.2ex}}
{\setlength{\parsep}{0.5ex plus 0.2ex}}
{\setlength{\labelwidth}{10mm}}
}{\end{list}}

\newcommand{\setnewcounter} {
\setcounter{subsection}{0}
\setcounter{equation}{0}
\setcounter{conjecture}{0}
\setcounter{assumption}{0}
\setcounter{question}{0}
\setcounter{definition}{0}
\setcounter{theorem}{0}
\setcounter{corollary}{0}
\setcounter{lemma}{0}
\setcounter{proposition}{0}
\setcounter{remark}{0}
}

\begin{document}
\title{\bf \Large Martingale approach for tail asymptotic problems in the generalized Jackson network}

\author{Masakiyo Miyazawa\\ Tokyo University of Science}
\date{November 6, 2017,\\ to appear in \normalsize {\it Probability and Mathematical Statistics} 37, No.\ 2}

\maketitle

\begin{abstract}
We study the tail asymptotic of the stationary joint queue length distribution for a generalized Jackson network (GJN for short), assuming its stability. For the two station case, this problem has been recently solved in the logarithmic sense for the marginal stationary distributions under the setting that inter-arrival and service times have phase-type distributions. In this paper, we study similar tail asymptotic problems on the stationary distribution, but problems and assumptions are different. First, the asymptotics are studied not only for the marginal distribution but also the stationary probabilities of state sets of small volumes. Second, the interarrival and service times are generally distributed and light tailed, but of phase type in some cases. Third, we also study the case that there are more than two stations, although the asymptotic results are less complete. For them, we develop a martingale method, which has been recently applied to a single queue with many servers by the author.
\end{abstract}

\section{Introduction}
\label{sect:introduction}

Asymptotic analyses have been actively studied in the recent queueing theory. This is because queueing models, particularly, queueing networks, become very complicated and their exact analyses are getting harder. We are interested in asymptotic analyses for large queues in a generalized Jackson network and aim to understand their asymptotic behaviors through its modeling primitives.

There are two different types of asymptotic analyses for large queues. One is for a given model fixed. Large deviations is typically studied for this. Another is to study them through an approximating model. For example, such a model is obtained as the limit of a sequence of models under heavy traffic by scaling of time, space and/or modeling primitives. It is called heavy traffic approximation (e.g., see \cite{Reim1984,Whit2002}). Here, large queues are caused by heavy traffic. In this paper, we focus on the large deviations for a fixed model. Among them, we are particularly interested in the logarithmic tail asymptotics of the stationary distribution for a generalized Jackson network, GJN for short.

This problem has been studied by the standard approach of large deviations, but the decay rates are hard to analytically get using modeling primitives (e.g., see \cite{Maje2009}). The author \cite{Miya2015a} recently studied it by a matrix analytic method, and derived the decay rates for the marginal stationary distributions in an arbitrary direction for a two station GJN, assuming phase type distributions for service times and arrival processes, called a phase-type setting. We aim to generalize this result under a more general setting by a different approach.

Let $d \ge 2$ be the number of stations in the GJN. For $d=2$, we relax the phase type assumption, and consider the decay rates of the stationary probabilities for state sets of small volumes in addition to those of the marginal stationary distribution. For $d \ge 3$, we derive upper and lower bounds for those decay rates.

Our basic idea is to simplify the derivation of those asymptotic results in such a way that they are obtained in a similar manner to a reflecting random walk on a multidimensional orthant. This simplification greatly benefits for analysis although the decay rate problems for the reflecting random walk have not been fully solved even for $d=2$. To this end, we take an approach studied for a single queue with heterogeneous servers in \cite{Miya2017}, and modify it for a queueing network. In this approach, we first describe the GJN by a piecewise deterministic Markov process, PDMP for short. We then derive martingales for change of measures, and formulate the asymptotic problems under a new measure. The idea for simplification is used in deriving the martingale.

PDMP is a continuous time Markov process whose sample path is composed of two parts, a continuous part, which is deterministic, and a discontinuous part, called jumps, by which randomness is created. Thus, PDMP is particularly suitable for queueing models. However, jump instants are random, and state changes at them are complicated. Because of this, PDMP is hard for analysis. So, other methods have been employed in queueing theory. For example, the state space is discretized using phase type distributions, and a Markov chain is obtained. Then, matrix analysis is applicable. This phase type approach is numerically powerful but analytically less explicit because of state space description. Furthermore, it is getting harder to apply as a queueing model becomes complicated like a queueing network. We will not use such a matrix analysis. Nevertheless, it turns out that the phase type assumption is helpful in our asymptotic analysis in some cases.

Contrary to the analytical difficulty, the PDMP has a simple sample path. Its time evolution is easily presented by a stochastic integral equation using a test function, which maps the states of the PDMP to real values (see \eq{dynamics 1}). In this stochastic equation, state changes at the jump instants cause difficulty for analysis as we mentioned above. \citet{Davi1984} who introduced PDMP replaces those state changes at jump instants by a martingale and the so called boundary condition on the test function.

However, it is not easy to find a good class of the test functions which characterize a distribution on the state space of the PDMP. The idea of \cite{Miya2017} is to choose a smaller class of test functions to overcome those difficulties. We then have a semi-martingale, which can not characterize a distribution on the state space, but still retains full information to study large queues. Once the semi-martingale is obtained, then we use the standard technique for change of measure through constructing an exponential martingale, called a multiplicative functional.

In applying this martingale approach to the GJN, we need to know how the network model is changed under the new measure. Intuitively, some of its stations must be unstable for the tail asymptotic analysis to work. To study this instability problem, we will use the fact that the network structure is unchanged under the change of measure, and therefore the stability of each station is characterized by the traffic intensity at that station. These traffic intensities are obtained from the traffic equations, but they are non-linear because of unstable stations. Thus, this instability problem is not obvious. We challenge it, and find some sufficient conditions for the GJN to be partly unstable under the new measure, which depends on the choice of a martingale for change of measure.

This paper is made up by four sections. In \sectn{GJN}, the GJN (generalized Jackson network) is described by a PDMP, and a martingale for change of measure is derived. This section also considers geometric interpretations of the stability condition of the GJN, and present main results for the asymptotic problems. \sectn{change} discusses the method of change of measure, and considers how the GJN is changed under a new measure. In \sectn{proofs}, the main results are proved. For this, we first list major steps for deriving upper and lower bounds, then prepare several lemmas to complete the proofs.

In this paper, we will use real vectors in the following way. Column and row vectors and their dimensions are not specified as long as they can be identified in the context where they are used. Their inequality holds in entry-wise. $\vcn{e}_{k}$ is the unit vector whose $k$-th entry is 1 while all the other entries vanish. $\vc{1}$ is the vector all of whose entries are 1. The inner product $\sum_{i} x_{i} y_{i}$ of vectors $\vc{x}, \vc{y}$ of the same dimension is denoted by $\br{\vc{x}, \vc{y}}$, and $\|\vc{x}\| = \sqrt{\br{\vc{x},\vc{x}}}$. $\vc{x}$ is said to be a unit direction vector if $\vc{x} \ge 0$ and $\|\vc{x}\| = 1$. We denote the set of all unit direction vectors in $\dd{R}^{d}_{+}$ by $\overrightarrow{U}_{d}$. For $\vc{x}$ in a finite dimensional vector space $\sr{S}$ and its subset $B$, we will use the convention that $\vc{x} + B = \{\vc{x} + \vc{y} \in \sr{S}; \vc{y} \in B\}$. For a finite set $A$, its cardinality is denoted by $|A|$.

\subsection*{Acknowledgement}

This paper is dedicated to Professor Tomasz Rolski for his great contribution to academic society. In particular, he organized a series of international conferences in Karpacz and Bedlewo in Poland since 1980 up to the last year. I have benefited from those conferences about not only academic collaborations but also personal interactions. This paper is partly supported by JSPS KAKENHI Grant Number 16H027860001.

\section{Generalized Jackson network}
\label{sect:GJN}

We are concerned with a queueing network which has a finite number of stations with single servers and single class of customers. At each station, there is an infinite buffer, exogenous customers arrive subject to a renewal process if any, and customers are served in First-Come-First-Served manner by independent and identically distributed service times. Furthermore, the renewal process and service times are independent of everything else. Customers who complete service at a station are independently routed to the next stations or leave the network according to a given probability. We refer this queueing network as a GJN (generalized Jackson network). 

\subsection{Notations and assumptions}
\label{sect:notation}

Let us introduce notations for a GJN. Let $d$ be the total number of stations. We index stations by elements in $\sr{J} \equiv \{1,2,\ldots,d\}$, and let $\sr{E}$ be the set of the stations which have exogenous arrivals. For each station $i$, let $F_{e,i}$ for $i \in \sr{E}$ be the interarrival time distribution of exogenous customers, and let $F_{s,i}$ for $i \in \sr{J}$ be the service time distribution. Let $p_{ij}$ be the probability that a customer completing service at station $i$ is routed to station $j$ for $i,j \in \sr{J}$, where those customer leave the outside of the network with probability:
\begin{align*}
   p_{i0} \equiv 1 - \sum_{i \in \sr{J}} p_{ij}. 
\end{align*}
To exclude trivial cases, we assume that $d \times d$ matrix $P \equiv \{p_{ij}; i,j \in \sr{J}\}$ is strictly substochastic, and $d+1 \times d+1$ matrix $\ol{P} \equiv \{p_{ij}; i,j \in \{0\} \cup \sr{J}\}$ is irreducible, where $p_{00} = 0$, and $p_{0i} > 0$ only if $i \in \sr{E}$, where the value of $p_{0i}$ is specified later. We call $P$ as a routing matrix, while $\ol{P}$ is called an over all routing matrix.

At time $t$, let $L_{i}(t)$ be the number of customers in station $i \in \sr{J}$, and let $R_{s,i}(t)$ be the residual service time of a customer being served there if any, where we set up a new service time just after service completion and this service time is unchanged as long as station $i$ is empty. Thus, all $R_{s,i}(t)$ are always positive because of the right continuity, and $R_{s,i}(t-)$ vanishes only at service completion instants. For $i \in \sr{E}$, let $R_{e,i}(t)$ be the residual time to the next exogenous arrival at station $i$.

Denote the vectors whose $i$-th entries are $L_{i}(t)$, $R_{s,i}(t)$ for $i \in \sr{J}$ and $R_{e,i}(t)$ for $i \in \sr{E}$ by $\vc{L}(t), \vc{R}_{s}(t), \vc{R}_{e}(t)$, respectively, and define $X(t)$ and $\vc{R}(t)$ as
\begin{align*}
  X(t) = (\vc{L}(t), \vc{R}_{e}(t), \vc{R}_{s}(t)), \qquad \vc{R}(t) = (\vc{R}_{e}(t), \vc{R}_{s}(t)), \qquad t \ge 0.
\end{align*}
Then, $\{p_{ij}; i,j \in \sr{J}\}$, $\{F_{e,i}; i \in \sr{E}\}$ and $\{F_{s,i}; i \in \sr{J}\}$ are the modeling primitives, and the state space $S$ for $X(t)$ is given by
\begin{align*}
  S = \{(\vc{z},\vc{y}_{e}, \vc{y}_{s}); \vc{z} \in \dd{Z}_{+}^{d}, \vc{y}_{e} \in \dd{R}_{+}^{\sr{E}}, \vc{y}_{s} \in \dd{R}_{+}^{d}\},
\end{align*}
where $\dd{Z}_{+}$ and $\dd{R}_{+}$ are the sets of all nonnegative integers and all nonnegative real numbers, respectively. As usual, we assume that $X(t)$ is right-continuous and has left-hand limits. Let $\{\sr{F}_{t}; t \ge 0\}$ be a filtration generated by histories of all the sample paths of $X(\cdot)$, then $\sr{F}_{t}$ is right-continuous, and $\{X(t); t \ge 0\}$ is a $\sr{F}_{t}$-Markov process. 

Let $\widehat{F}_{e,i}$ and $\widehat{F}_{s,i}$ be the moment generating functions, MGF for short, of the distributions $F_{e,i}$ and $F_{s,i}$, respectively. We define $\beta_{w,i}$ and $\theta_{w,i}$ for $w=e,s$ as
\begin{align}
\label{eq:limit point 1}
 & \beta_{w,i} = \sup \{\theta \in \dd{R}; \widehat{F}_{w,i}(\theta) < \infty\}, \qquad \theta_{w,i} = \inf \{\theta \in \dd{R}; e^{-\theta} < \widehat{F}(\beta_{w,i}) \}
\end{align}
We will assume that $\beta_{w,i} > 0$ and $\theta_{w,i} = \infty$ for all $w=e,s$ and $i$. That is, all the distributions, $F_{w,i}$, have light tail and their moment generating functions diverges at their singular points. These conditions are assumed for technical simplicity.

For some important cases, we have to restrict these distributions in the following class. A positive random variable $T$ or its distribution is said to have a conditional MGF (moment generating function) with a uniform bound if there is a function $h$ of $\theta > 0$ such that $\dd{E}(e^{\theta T}) < \infty$ implies that 
\begin{align}
\label{eq:residual 1}
  \dd{E}(e^{\theta (T-t)})|T>t) \le h(\theta), \qquad \theta > 0,
\end{align}
as long as $\dd{P}(T>t) > 0$. Obviously, if $T$ is bounded, it satisfies \eq{residual 1}. Another obvious example is a NBU distribution, which is characterized by $\dd{P}(T> s+t|T>s) \le \dd{P}(T>t)$ for $s,t > 0$. An important class for our application is of phase type, which is defined as
\begin{align}
\label{eq:phase type 1}
  F(t) \equiv \dd{P}(T \le t) = 1 - \vc{a} e^{t U} \vc{1}, \qquad t \ge 0,
\end{align}
where $\vc{a}$ is a finite dimensional probability vector, and $U$ is a defective transition rate matrix with the same dimension as $\vc{a}$ such that $(-U)^{-1}$ is finite.

\begin{lemma}
\label{lem:phase type 1}
A phase type distribution has a conditional MGF with a uniform bound.
\end{lemma}
\begin{proof}
Assume that $F$ is given by \eq{phase type 1}. Let $T$ be a random variable subject to $F$, and let $\vc{b}(s) = \frac {\vc{a} e^{s U}} {\vc{a} e^{s U} \vc{1}}$, then $\vc{b}(s)$ is a probability vector, and
\begin{align*}
  \dd{P}(T > s+t|T>t) = \vc{b}(t) e^{s U} \vc{1}, \qquad s,t>0,
\end{align*}
and therefore
\begin{align*}
  \dd{E}(e^{\theta (T-t)})|T>t) = \vc{b}(t)(-U)^{-1} (\theta I + U)^{-1} \vc{1}, \qquad t \ge 0, \theta > 0,
\end{align*}
which is finite as long as $\dd{E}(e^{\theta T}) = \vc{a} (-U)^{-1} (\theta I + U)^{-1} \vc{1}$ is finite. Hence, we have \eq{residual 1} by letting $h(\theta)$ be the maximum of all the entries of the vector $(-U)^{-1} (\theta I + U)^{-1} \vc{1}$.
\end{proof}

Thus, we consider the tail asymptotic problem for the GJN assuming the distributions of $T_{e,i}, T_{s,j}$ to have light tails, and, in some cases, we assume:
\begin{mylist}{3}
\item [(A1)] All the $T_{e,i}$ for $i \in \sr{E}$ and $T_{s,i}$ for $j \in \sr{J}$ have conditional MGF with uniform bounds, that is satisfy \eq{residual 1}. 
\end{mylist}
Let $\lambda_{e,i} = 1/\dd{E}(T_{e,i})$ for $i \in \sr{E}$ and $\mu_{s,i} = 1/\dd{E}(T_{s,i})$ for $i \in \sr{J}$. For convenience, we put $\lambda_{e,i} = 0$ for $i \in \sr{J} \setminus \sr{E}$. Let $\alpha^{(0)}_{i}$ for $i \in \sr{J}$ be the solutions of the following traffic equation.
\begin{align}
\label{eq:traffic 1}
  \alpha^{(0)}_{i} = \lambda_{i} + \sum_{j\in \sr{J}} \alpha^{(0)}_{j} p_{ji}, \qquad i \in \sr{J}.
\end{align}
It is easy to see that the solutions uniquely exist by the strict substochastic of the routing matrix $P$ and the irreducibility of $\ol{P}$, where we now put $p_{0i} = \lambda_{i}/\sum_{j \in \sr{J}} \lambda_{j}$ for $i \in \sr{J}$. Let $\rho^{(0)}_{i} = \alpha^{(0)}_{i} \dd{E}(T_{s,i})$, and assume the stability condition that
\begin{align}
\label{eq:stability 1}
  \rho^{(0)}_{i} < 1, \qquad i \in \sr{J}.
\end{align}
In \sectn{stability}, we will consider the case that some of stations are unstable. This case occurs under change of measure, and $\rho^{(0)}_{i}$ is no longer a right traffic intensity. This is the reason why we put superscript ``$^{(0)}$'' here.

\subsection{Piecewise Deterministic Markov process, PDMP}
\label{sect:PDMP}

In this paper, we consider $\{X(t); t \ge 0\}$ as a piecewise deterministic Markov process, PDMP for short, introduced by \citet{Davi1984}. PDMP is a Markov process with piece-wise deterministic and continuously differentiable sample path and finitely many discontinuous epochs in each finite time interval. Its randomness arises at discontinuous epochs, which are uniquely determined by hitting times when the deterministic sample path gets into a specified state set. The set of those discontinuous epochs constitute a counting process, and the piece-wise deterministic sample path is randomly changed at those times. We here assume that there is no other discontinuous state change. This slightly changes the standard description of PDMP due to \citet{Davi1984}, but it is a matter of formulation since Davis' PDMP can be described by the present formulation as well.

We now introduce notations to describe $X(t)$ as a PDMP. Let $N$ be a counting process for the expiring times of all the remaining times. That is,
\begin{align*}
  N(t) = \sum_{u \in (0,t]} \Big( \sum_{i \in \sr{E}} 1(\Delta R_{e.i}(u) > 0) + \sum_{i \in \sr{J}} 1(\Delta R_{s.i}(u) > 0) \Big), \qquad t \ge 0,
\end{align*}
where $\Delta$ is the difference operator such that $\Delta f(t) = f(t) - f(t-)$ for a function $f$ which is right-continuous and has left-hand limits. Clearly, $N$ counts all the discontinuous points of $X(t)$. However, it may multiply counts at the same instant, and therefore $\Delta N(t)$ may be greater than 1. To avoid this, we define a simplification of $N$ as
\begin{align*}
  N^{*}(t) = \sum_{u \in (0,t]} 1(\Delta N(u) > 0), \qquad t \ge 0.
\end{align*}
We then let $t_{0} = 0$, and inductively define $t_{n} = \inf\{u > t_{n-1}; \Delta N^{*}(u) > 0\}$ for $n=1, 2, \ldots$. Thus, $t_{n}$ is the $n$-th discontinuous epoch of $X(t)$, and a stopping time with respect to $\sr{F}_{t}$.

Between times $t_{n-1}$ and $t_{n}$, $X(t)$ is linearly changes, so continuously differentiable in such way that
\begin{align*}
  \frac d{dt} L_{i}(t) = 0, \qquad \frac d{dt} R_{e,i}(t) = -1(i \in \sr{E}), \qquad \frac d{dt} R_{s,i}(t) = 1(R_{s,i}(t) > 0), \qquad i \in \sr{J}.
\end{align*}
This differentiation can be described by an operator $\sr{A}$ on $C^{1}(S)$, which is the set of all continuously differentiable functions from $S$ to $\dd{R}$. Namely, $\sr{A}$ is defined as
\begin{align}
\label{eq:A 1}
  \sr{A}f(\vc{x}) = - \sum_{i \in \sr{E}} \frac {\partial} {\partial y_{e,i}} f(\vc{z},\vc{y}_{e},\vc{y}_{s}) - \sum_{i \in \sr{J}} \frac {\partial} {\partial y_{s,i}} f(\vc{z},\vc{y}_{e},\vc{y}_{s}) 1(z_{i} \ge 1).
\end{align}

Since PDMP is a strong Markov process, the conditional distribution of $X(t_{n})$ given $\sr{F}_{t_{n}-}$ is a function of $X(t_{n}-)$ for each $n \ge 1$, which is characterized by the transition kernel $\sr{K}$ given below.
\begin{align}
\label{eq:K 1}
  \sr{K} f \big(X(t-) \big) = \dd{E}\big( f(X(t)) | X(t-)\big), \quad X(t-) \in \Gamma,
\end{align}
for $f \in \sr{M}(S)$, where $\Gamma$ is the set of $\vc{x} \equiv (\vc{z},\vc{y}_{e},\vc{y}_{s}) \in S$ such that
\begin{align*}
  \exists i \in \sr{E}, y_{e,i} = 0 \quad \mbox{or} \quad \exists i \in \sr{J}, z_{i} \ge 1, y_{s,i} = 0.
\end{align*}
This $\Gamma$ is referred to as a terminal set, while $\sr{K}$ is referred to as a jump kernel.

\subsection{Martingale decomposition of the PDMP} 
\label{sect:martingale}

From \eq{A 1} and \eq{K 1} and the counting process $N^{*}$, we have a time evolution equation.
\begin{align}
\label{eq:dynamics 1}
  f(X(t)) = & f(X(0))  + \int_{0}^{t} \sr{A} f(X(u)) du + \int_{0}^{t} \Delta f(X(u)) dN^{*}(u), \quad f \in C^{1}(S).
\end{align}
We refer to $f$ as a test function as is typically called.

We apply the same martingale method as discussed in \cite{Miya2017}. We here repeat them briefly for this paper to be self-content. We first note that
\begin{align*}
  M_{0}(t) \equiv \int_{0}^{t} (f(X(u)) - \sr{K}f(X(u-))) dN^{*}(u), \qquad t \ge 0,
\end{align*}
is an $\sr{F}_{t}$-martingale if $\dd{E}(|M(t)|) < \infty$. Since
\begin{align*}
  \Delta f(X(u)) = f(X(u)) - \sr{K}f(u-) + \sr{K}f(X(u-)) - f(X(u-)),
\end{align*}
it follows from \eq{dynamics 1} that
\begin{align}
\label{eq:martingale 2}
 M_{0}(t) = f(X(t)) - f(X(0))  - \Big(&\int_{0}^{t} \sr{A} f(X(s)) ds\nonumber\\
 & + \int_{0}^{t} (\sr{K}f(X(s-)) - f(X(s-))) dN^{*}(s)\Big).
\end{align}
We define $M(\cdot)$ and $A(\cdot)$ as
\begin{align}
\label{eq:martingale 3}
 & M(t) = f(X(t)) - f(X(0)) - \int_{0}^{t} \sr{A} f(X(s)) ds, \qquad t \ge 0,\\
\label{eq:BV 1}
 & A(t) = \int_{0}^{t} (\sr{K}f(X(s-)) - f(X(s-))) dN^{*}(s).
\end{align}
Since
\begin{align}
\label{eq:M 1}
  M(t) = M_{0}(t) + A(t), \qquad t \ge 0,
\end{align}
we have the following fact.

\begin{lemma}
\label{lem:martingale 1}
For the PDMP $X(\cdot)$, if the condition:
\begin{align}
\label{eq:terminal 1}
  \sr{K}f(\vc{x}) = f(\vc{x}), \qquad \forall \vc{x} \in \Gamma,
\end{align}
is satisfied and if $\dd{E}(|M(t)|) < \infty$ for all $t \ge 0$, then $M(\cdot)$ is an $\sr{F}_{t}$-martingale. In particular, if \eq{terminal 1} with equality holds, then $M(\cdot)$ is an $\sr{F}_{t}$-martingale.
\end{lemma}

We refer to \eq{terminal 1} as a terminal condition following the terminology of \cite{Miya2017}.

\subsection{Terminal condition for the GJN}
\label{sect:terminal}

A key of our arguments is to find a set of test functions satisfying the terminal condition \eq{terminal 1}. For this, we mainly use the following test function, parameterized by $\vc{\theta} \in \dd{R}^{d}$. 
\begin{align}
\label{eq:test 1}
  f_{\vc{\theta}}(\vc{x}) = e^{\br{\vc{\theta}, \vc{z}} - \br{\vc{\gamma}_{e}(\vc{\theta}), \vc{y}_{e}} - \br{\vc{\gamma}_{s}(\vc{\theta}), \vc{y}_{s}}}, \qquad \vc{x} \equiv (\vc{z},\vc{y}_{e},\vc{y}_{s}) \in S,
\end{align}
using some vector valued functions $\vc{\gamma}_{e}(\vc{\theta}) \in \dd{R}^{\sr{E}}$ and $\vc{\gamma}_{s}(\vc{\theta}) \in \dd{R}^{\sr{J}}$, where we recall that $\br{\vc{a}, \vc{b}}$ is the inner product of vectors $\vc{a}, \vc{b}$ of the same dimensions. In some cases, it needs to truncate some of $y_{e,i}$ and $y_{s,j}$ as $y_{e,i} \wedge v$ and $y_{s,j} \wedge v$ for $v > 0$, which causes to change $\gamma_{e,i}(\theta_{i})$ to $\gamma_{e,i}(v,\theta_{i})$ as we will see, where $a \wedge b = \min(a,b)$ for $a, b \in \dd{R}$. By $J_{e}(v)$, we denote the set of $i \in \sr{E}$ such that $y_{e,i}$ is truncated by $v$. Similarly, $J_{s}(v)$ denotes the set of $i \in \sr{J}$ such that $y_{s,i}$ is truncated by $v$ for $i \in J_{s}(v)$. Let $\vc{J}(v) = (J_{e}(v), J_{s}(v)) \subset \sr{E} \times \sr{J}$. Then, the test function $f_{\vc{\theta}}$ is changed to
\begin{align}
\label{eq:test 2}
  f_{\vc{J}(v),\vc{\theta}}(\vc{x}) = e^{\br{\vc{\theta}, \vc{z}} - w_{\vc{J}(v)}(\vc{\theta}, \vc{y})}, \qquad \vc{x} \equiv (\vc{z},\vc{y}) \in S,
\end{align}
where $\vc{y} = (\vc{y}_{e}, \vc{y}_{s})$ and
\begin{align}
\label{eq:w J 1}
  w_{\vc{J}(v)}(\vc{\theta}, \vc{y}) & = \sum_{i \in J_{e}(v)} \gamma_{e.i}(v,\theta_{i}) (y_{e,i} \wedge v) + \sum_{i \in \sr{E} \setminus J_{e}(v)} \gamma_{e,i}(\theta_{i}) y_{e,i} \nonumber\\
  & \quad + \sum_{i \in J_{s}(v)} \gamma_{s,i}(v,\theta_{i}) (y_{s,i} \wedge v) + \sum_{i \in \sr{J} \setminus J_{e}(v)} \gamma_{s,i}(\theta_{i}) y_{s,i}.
\end{align}
Obviously, $f_{\emptyset,\vc{\theta}}(\vc{x}) = f_{\vc{\theta}}(\vc{x})$.

Our first task is to determine functions $\gamma_{e,i}(v,\cdot)$ and $\gamma_{s,j}(v,\cdot)$ for $v \in (0,\infty]$ so that the terminal condition \eq{terminal 1} is satisfied, where $\gamma_{u,i}(\cdot) = \gamma_{u,i}(\infty,\cdot)$ for $u=e,s$. For this, we first consider a prototype for them as we have done in Section 2.3 of \cite{Miya2017}. Let $T$ be a positive valued random variable, and denote its distribution by $F$. We truncate $T$ by a positive number $v$ as $T^{(v)} \equiv T \wedge v$, and denote the distribution of $T^{(v)}$ by $F^{(v)}$. We denote the moment generating functions of $F$ and $F^{(v)}$ by $\widehat{F}$ and $\widehat{F}^{(v)}$, respectively.

Note that $\widehat{F}^{(v)}(\theta)$ exists and finite for all $\theta \in \dd{R}$, but this may not be true for $\widehat{F}(\theta) \equiv \widehat{F}^{(\infty)}(\theta)$. For $v \in (0,\infty]$, let
\begin{align}
\label{eq:beta 1}
  \beta_{F}^{(v)} = \sup\{\theta \in \dd{R}; \widehat{F}^{(v)}(\theta) < \infty\}, \qquad \theta_{F}^{(v)} = \inf \{\theta \in \dd{R}; e^{-\theta} < \widehat{F}^{(v)}(\beta_{F}^{(v)}) \}.
\end{align}
then $\beta_{F}^{(v)} = \infty$ and $\theta_{F}^{(v)} = - \infty$ for $v < \infty$, while they may be finite for $v= \infty$, where $\theta_{F}^{(v)} \le 0$ since $\beta_{F}^{(v)} \ge 0$. Note that $\widehat{F}(\theta)$ is finite for $\theta < \beta_{F}^{(\infty)}$. Define $\xi$ be a solution of the following equation.
\begin{align}
\label{eq:xi 1}
  e^{\theta} \widehat{F}^{(v)}(\xi) = 1, \qquad \theta \in \dd{R}, v > 0.
\end{align}
Obviously, $\xi$ uniquely exists for each $\theta \in \dd{R}$ and $v>0$. We denote it by $\xi_{F}(v,\theta)$. It has the following properties, which are proved in Lemma 2.4 of \cite{Miya2017}.
\begin{lemma}
\label{lem:xi 1}
For each fixed $v > 0$, \\
(a) $\xi_{F}(v,0) = 0$, and $\xi_{F}(v,\theta)$ is strictly decreasing and concave in $\theta \in \dd{R}$.\\
(b) $\xi_{F}(v,\theta)$ is positive and decreasing in $v$ for each fixed $\theta < 0$.\\
(c) $\xi_{F}(v,\theta)$ is negative increasing in $v$ for each fixed $\theta > 0$.\\
(d) $\xi_{F}(v,\theta)$ is differentiable in $\theta$, and
\begin{align*}
  \frac {\partial}{\partial \theta} \xi_{F}(v,\theta) = - \frac {e^{-\theta}} {(\widehat{F}^{(v)})'(\xi_{F}(v,\theta))}.
\end{align*}
\end{lemma}

We define $\xi_{F}^{(\vartriangle)}(\theta)$ and $\xi_{F}(\theta)$ as
\begin{align*}
  \xi_{F}^{(\vartriangle)}(\theta) = \lim_{v \uparrow \infty} \xi_{F}(v,\theta), \quad \theta \in \dd{R}, \qquad \xi_{F}(\theta) = \xi_{F}(+\infty,\theta), \quad \theta > \theta_{F}^{(\infty)},
\end{align*}
which exist and are finite. These functions have some nice properties. For them, we cite Lemma 2.5 of \cite{Miya2017} in which $\ol{\theta} = \infty$ in the present case.

\begin{lemma}
\label{lem:concave 1}
(a) $\xi_{F}^{(\vartriangle)}(\theta)$ is nonincreasing and concave for all $\theta \in \dd{R}$. (b) 
\begin{align}
\label{eq:xi 2}
  \xi_{F}^{(\vartriangle)}(\theta) = \left\{
\begin{array}{ll}
 \beta_{F}^{(\infty)}, & \theta \le \theta_{F}^{(\infty)}, \\
 \xi_{F}(\theta) & \theta > \theta_{F}^{(\infty)},
\end{array}
\right. \qquad
  \frac d{d \theta} \xi_{F}^{(\vartriangle)}(\theta) = \left\{
\begin{array}{ll}
 0, & \theta < \theta_{F}^{(\infty)}, \\
 (\xi_{F}^{(\infty)})'(\theta) & \theta > \theta_{F}^{(\infty)},
\end{array}
\right.\end{align}
where $(\xi_{F}^{(\infty)})'(\theta_{F}^{(\infty)})$ is the derivative from the right if $\theta_{F}^{(\infty)}$ is finite, and
\begin{align}
\label{eq:xi 3}
  e^{\theta} \widehat{F}(\xi_{F}^{(\vartriangle)}(\theta)) \begin{cases}
 \le 1, \quad & \theta < \theta_{F}^{(\infty)} \\
 = 1, &  \theta \ge \theta_{F}^{(\infty)}.
\end{cases}
\end{align}
\end{lemma}

Throughout paper, we assume that
\begin{align}
\label{eq:light tail 1}
  \theta^{(\infty)}_{F_{e,i}}  = - \infty, \quad i \in \sr{E}, \qquad \theta^{(\infty)}_{F_{s,i}} = - \infty, \quad i \in \sr{J},
\end{align}
which means that $F_{e,i}$ and $F_{s,j}$ have light tails and their moment generating functions diverges at the upper boundaries of their convergence domains. This assumption can be removed using $\xi_{F}^{(\vartriangle)}$ as shown in \cite{Miya2017} for a single queue. However, it will be complicated for a queueing network. This is the reason why we assume \eq{light tail 1}.

Let
\begin{align*}
  q_{i}(\vc{\theta}) = e^{-\theta_{i}} \Big(\sum_{j \in \sr{J}} p_{ij} e^{\theta_{j}} + p_{i0}\Big), \qquad \vc{\theta} \in \dd{R}^{d}, i \in \sr{J},
\end{align*}
then $q_{i}(\vc{\theta}) > 0$ and is convex in $\vc{\theta}$, and it is easy to see that
\begin{align*}
  \log q_{i}(\vc{\theta}) = - \theta_{i} + \log \Big(\sum_{j \in \sr{J}} p_{ij} e^{\theta_{j}} + p_{i0}\Big)
\end{align*}
is a convex function of $\vc{\theta} \in \dd{R}^{d}$ because $\sum_{j \in \sr{J}} p_{ij} e^{\theta_{j}} + p_{i0}$ is a sum of convex functions (see Lemma of \cite{King1961b}).

We now define $\gamma_{e,i}(v,\theta_{i})$ and $\gamma_{s,i}(v,\vc{\theta})$ for $v \in (0,\infty]$ as
\begin{align}
\label{eq:gamma 1}
 & \gamma_{e,i}(v,\theta_{i}) = - \xi_{F^{(v)}_{e,i}}(\theta_{i}), \quad i \in \sr{E}, \qquad \gamma_{s,i}(v,\vc{\theta}) = - \xi_{F^{(v)}_{s,i}}(\log q_{i}(\vc{\theta})), \quad i \in \sr{J}.
\end{align}
As informally mentioned, we let $\gamma_{e,i}(\theta_{i}) = \gamma_{e,i}(\infty,\theta_{i})$ and $\gamma_{s,i}(\vc{\theta}) = \gamma_{s,i}(\infty,\vc{\theta})$. Due to the assumption \eq{light tail 1}, these functions are well defined for all $\vc{\theta} \in \dd{R}^{d}$. Clearly, their definitions are equivalent to:
\begin{align}
\label{eq:gamma 2}
  e^{\theta_{i}} \widehat{F}^{(v)}_{e,i}(-\gamma_{e,i}(v,\theta_{i})) = 1, \quad i \in \sr{E}, \qquad q_{i}(\vc{\theta}) \widehat{F}^{(v)}_{s,i}(-\gamma_{s,i}(v,\vc{\theta})) = 1, \quad i \in \sr{J}.
\end{align}
These equations mean that $\Delta R_{e,i}(t)$ and $\Delta R_{s,j}(t)$ at the jump instants are compensated by the change of the queue lengths so that the terminal condition \eq{terminal 1} is satisfied. This is an intuitive background for the definitions of $\gamma_{e,i}, \gamma_{s,j}$. 

\begin{remark}
\label{rem:gamma 1}
The reader may wonder why the minus signs are needed in \eq{gamma 1} because $\gamma_{e,i}, \gamma_{s,j}$ in the test functions $f_{\vc{\theta}}$ and $f_{\vc{\theta}}$ also have the minus signs and they can be cancelled. The reason for this is that they have nice interpretations for large deviations. For example, let $N_{e,i}(t)$ be the number of arrivals at station $i \in \sr{E}$ by time $t$, then $N_{e,i}(\cdot)$ is a renewal process, and \citet{GlynWhit1994a} show that
\begin{align}
\label{eq:LD 1}
  \lim_{t \to \infty} \frac 1t \log \dd{E}(e^{\theta_{i} N_{e,i}(t)}) = \gamma_{e,i}(\theta_{i}), \qquad \theta_{i} > \theta_{F_{e,i}}^{(\infty)},
\end{align}
for any initial distribution for $N_{e,i}(\cdot)$. This suggests that $\gamma_{e,i}(\theta_{i})$ must be one of key information for the tail asymptotic of our problem. However, we will not use this property of $\gamma_{e,i}$ because the definition \eq{gamma 1} is sufficiently informative for our analysis.
\end{remark}

Note that $\gamma_{e,i}(v,\theta_{i})$ and $\gamma_{s,i}(v,\vc{\theta})$ are convex in $\theta_{i}$ and $\vc{\theta}$, respectively, because $\xi_{F^{(v)}_{e,i}}(\theta)$ and $\xi_{F_{s,i}}(\vc{\theta})$ are decreasing and concave in $\theta \in \dd{R}$ and $\log q_{i}(\vc{\theta})$ is convex. For $v \in (0,\infty]$, and $\vc{J}(v) \equiv (\vc{J}(v),J_{s}(v)) \subset \sr{E} \times \sr{J}$, let, for $\vc{\theta} \in \dd{R}^{d}$,
\begin{align}
\label{eq:gamma 3}
 \gamma_{\vc{J}(v)}(\vc{\theta}) & = \sum_{i \in \vc{J}(v)} \gamma_{i}(v,\theta_{i}) + \sum_{i \in \sr{E} \setminus \vc{J}(v)} \gamma_{i}(\theta_{i}) + \sum_{i \in J_{s}(v)} \gamma_{i}(v,\theta_{i}) + \sum_{i \in \sr{J} \setminus J_{s}(v)} \gamma_{i}(\theta_{i}),
\end{align}
and ${\gamma}(\vc{\theta}) = \gamma_{\emptyset}(\vc{\theta})$, that is, 
\begin{align*}
   & {\gamma}(\vc{\theta}) = \sum_{i \in \sr{E}} \gamma_{e,i}(\theta_{i}) + \sum_{i \in \sr{J}} \gamma_{s,i}(\vc{\theta}).
\end{align*}
Furthermore, $\gamma_{\vc{J}(v)}(\vc{\theta})$ converges to ${\gamma}(\vc{\theta})$ for each $\vc{\theta} \in \dd{R}^{d}$ as $v \to \infty$, which is uniform on a compact set of $\vc{\theta}$. The next lemma is a key for our arguments, and easily follows from Lemma 3.2 in \cite{BravDaiMiya2015}. We also remarked its intuitive meaning below \eq{gamma 2}. So far, its proof is omitted.

\begin{lemma}
\label{lem:terminal 1}
For $v \in (0,\infty]$, test function $f_{\vc{J}(v),\vc{\theta}}$ of \eq{test 2} satisfies the terminal condition \eq{terminal 1} with equality for all $\vc{\theta} \in \dd{R}^{d}$, respectively.
\end{lemma}

We next consider a martingale for the test functions $f_{\vc{J}(v),\vc{\theta}}$. Denote the probability measure for $X(\cdot)$ with the initial state $\vc{x} \in S$ by $\dd{P}_{x}$, and let $\dd{E}_{x}$ stand for the expectation under $\dd{P}_{x}$. We first note that
\begin{align}
\label{eq:finite 1}
  \dd{E}_{\vc{x}}(f_{\vc{J}(v),\vc{\theta}}(X(t))) < \infty, \qquad t \ge 0,
\end{align}
always holds for each $\vc{x} \in S$ and $\vc{\theta} \in \dd{R}^{d}$ because the total number of exogenous arrivals and service completions  in each finite time interval has a super-light tail (lighter than any exponential decay) (see, e.g.,  Lemma 4.1 of \cite{Miya2017} for the single queue case). Hence, Lemmas \lemt{martingale 1} and \lemt{terminal 1} immediately imply the following fact.

\begin{lemma}
\label{lem:martingale 2}
Fix $\vc{\theta} \in \dd{R}^{d}$ and $\vc{x} \in S$. For the PDMP $X(\cdot)$ and test function $f_{\vc{\theta}}$ of \eq{test 2}, let
\begin{align}
\label{eq:martingale 4}
 M_{\vc{J}(v),\vc{\theta}}(t) & =  f_{\vc{J}(v),\vc{\theta}}(X(t)) - f_{\vc{J}(v),\vc{\theta}}(X(0)) + \int_{0}^{t} {\gamma}(\vc{\theta}) f_{\vc{J}(v),\vc{\theta}}(X(u) du \nonumber\\
 & \quad - \sum_{i \in J_{e}(v)} \gamma_{e,i}(v,\theta_{i}) \int_{0}^{t} 1(R_{e,i}(u) > v) f_{\vc{J}(v),\vc{\theta}}(X(u)) du\nonumber\\
 & \quad - \sum_{i \in J_{s}(v)} \gamma_{s,i}(v,\vc{\theta}) \int_{0}^{t} 1(R_{s,i}(u) > v) f_{\vc{J}(v),\vc{\theta}}(X(u)) du\nonumber\\
 & \quad - \int_{0}^{t} \sum_{i \in J_{e}(v)} (\gamma_{s,i}(v,\vc{\theta}) 1(L_{i}(u) = 0) f_{\vc{J}(v),\vc{\theta}}(X(u) du \nonumber\\
 & \quad - \int_{0}^{t} \sum_{i \in \sr{J} \setminus J_{e}(v)} \gamma_{s,i}(\vc{\theta}) 1(L_{i}(u) = 0)  f_{\vc{J}(v),\vc{\theta}}(X(u)) du, \qquad t \ge 0,
\end{align}
then $M_{\vc{J}(v),\vc{\theta}}(\cdot)$ is an $\sr{F}_{t}$-martingale under $\dd{P}_{\vc{x}}$.
\end{lemma}

As always, $M_{\vc{\theta}}(\cdot)$ with $v=\infty$ is simply denoted by $M_{\vc{\theta}}(\cdot)$, which also is an $\sr{F}_{t}$-martingale under $\dd{P}_{\vc{x}}$.  Note that \eq{martingale 4} may read as a semi-martingale representation of $f_{\vc{J}(v),\vc{\theta}}(X(t))$.
 
\subsection{Stability condition and geometric interpretation}
\label{sect:stability}

As we mentioned in \sectn{notation}, the GJN (generalized Jackson network) is stable if the stability condition \eq{stability 1} holds. Except for trivial cases, it is also necessary. We will consider this network under change of measure, which is generally unstable, and it is important to see under what condition which station is unstable. To make clear these arguments, we formally define stability and instability for each station. Station $i$ is said to be weakly stable (stable) if $L_{i}(t)$ is recurrent (positive recurrent, respectively), and to be weakly unstable (unstable) if $L_{i}(t)$ is null recurrent or transient (transient, respectively). 

In this subsection, we so far do not assume the stability condition \eq{stability 1}, and consider conditions for a station to be unstable (or stable). For this, we first need to compute an arrival rate at each station, which is obtained as the maximal solution $\{\alpha_{i}; i \in \sr{F}\}$ of the following traffic equation (e.g., see \cite{ChenMand1991,ChenMand1991b}).
\begin{align}
\label{eq:traffic 2}
  \alpha_{i} = \lambda_{i} + \sum_{j\in \sr{J}} (\alpha_{j} \wedge \mu_{j}) p_{ji}, \qquad i \in \sr{J},
\end{align}
where we recall that $\mu_{i} = 1/\dd{E}(T_{s,i})$. Let $\rho_{i} = \alpha_{i}/\mu_{i}$, which may be different from $\rho^{(0)}$ (see at the end of \sectn{notation}). Under appropriate conditions such as $T_{s,i}$ has a spread out distribution (see \cite{Asmu2003} for its definition), station $i$ is weakly stable (stable) if and only if $\rho_{i} \le 1$ ($\rho_{i} < 1$), and weakly unstable (unstable) if and only if $\rho_{i} \ge 1$ ($\rho_{i} > 1$, respectively). 

It is easy to see that $\alpha_{i} \le \alpha^{(0)}_{i}$ for all $i \in \sr{J}$, where recall that $\alpha^{(0)}_{i}$ is the solution of the standard traffic equation \eq{traffic 1}. The $\alpha_{i}$ can be numerically obtained from \eq{traffic 2} in finite steps, but it is hard to get its analytic expression. For us, it is particularly important to give sufficient conditions in terms of $\gamma_{e,i}, \gamma_{s,i}$ for a station to be unstable or weakly unstable because these functions are well handled under change of measure. We first give sufficient conditions for instability in terms of $\lambda_{i}, \mu_{i}, p_{ij}$ and $\alpha^{(0)}_{i}$.

\begin{lemma} \rm
\label{lem:unstable 1}
(a) For each $j \in \sr{J}$, if either $\alpha^{(0)}_{j} \le \mu_{j}$ or
\begin{align}
\label{eq:mu 1}
  \lambda_{j} + \sum_{k\in \sr{J}} \mu_{k} p_{kj} \le \mu_{j},
\end{align}
holds, then $\rho_{j} \le 1$. That is, station $j$ is weakly stable. \\
(b) If $\mu_{j} < \alpha^{(0)}_{j}$ and if $\rho_{i} \le 1$ for all $i \in \sr{J} \setminus \{j\}$, then $\rho_{j} > 1$, that is, station $j$ is unstable.\\
(c) If, for all $ j \in \sr{J}$,
\begin{align}
\label{eq:mu 2}
  \lambda_{j} + \sum_{k\in \sr{J}} \mu_{k} p_{kj} \ge \mu_{j}, 
\end{align}
then $\rho_{j} \ge 1$ for all $j \in \sr{J}$. That is, all stations is weakly unstable. If \eq{mu 2} holds with strict inequality for $j=i$, then $\rho_{i} > 1$, that is, station $i$ is unstable.
\end{lemma}

\begin{remark}
\label{rem:unstable 1}
For our application, it would be nice if (b) can be generalized in such a way that, for $A \subset \sr{J}$, if $\mu_{j} < \alpha^{(0)}_{j}$ for all $j \in A$ and if $\rho_{i} \le 1$ for all $i \in \sr{J} \setminus A$, then $\rho_{j} > 1$ for all $j \in A$. Unfortunately, this is generally not true. A counterexample is easily obtained, for example, for a three station tandem queue (see Section 4 of \cite{ChenMand1991} for some related discussions). We need to update $\alpha^{(0)}_{j}$ using the information on the unstable station to be available to get such a generalization, but it would be less analytically tractable. Thus, we will not pursue it in this paper.
\end{remark}

\begin{proof}
(a) Since $\alpha_{j} \le \alpha^{(0)}_{j}$, it follows from $\alpha^{(0)}_{j} \le \mu_{j}$ that $\rho_{j} \le 1$. If \eq{mu 1} holds, \eq{traffic 2} implies that $\alpha_{j} \le \mu_{j}$, which is equivalent to $\rho_{j} \le 1$. \\
(b) Suppose that $\rho_{j} \le 1$ contrary to the claim, then $\rho_{k} \le 1$ for all $k \in \sr{J}$ by the second assumption. Hence, $\alpha_{k} \le \mu_{k}$ for all $k \in \sr{J}$, and therefore the non-linear traffic equation \eq{traffic 2} is identical with the linear traffic equation \eq{traffic 1}. Thus, $\alpha^{(0)}_{k} = \alpha_{k} \le \mu_{k}$ for all $k \in \sr{J}$. This contradicts the assumption that $\mu_{i} < \alpha^{(0)}_{i}$, and therefore (b) is proved.\\
(c) Let $A = \{i \in \sr{J}; \alpha_{i} < \mu_{i}\}$, then \eq{traffic 2} can be written as
\begin{align*}
  \alpha_{j} = \lambda_{j} + \sum_{i \in A} \alpha_{i} p_{ij} + \sum_{i \in \sr{J} \setminus A} \mu_{i} p_{ij}.
\end{align*}
Hence, \eq{mu 2} implies that
\begin{align*}
  \alpha_{j} - \mu_{j} = \lambda_{j} - \mu_{j} + \sum_{i \in A} (\alpha_{i} - \mu_{i}) p_{ij} + \sum_{i \in \sr{J}} \mu_{i} p_{ij}  \ge \sum_{i \in A} (\alpha_{i} - \mu_{i}) p_{ij}.
\end{align*}
We then sum up both sides of this inequality for all $j \in A$, which yields
\begin{align*}
  \sum_{i \in A} (\alpha_{i} - \mu_{i}) \Big(1 - \sum_{j \in A} p_{ij}\Big) \ge 0.
\end{align*}
Since $\alpha_{i} - \mu_{i} < 0$ for $i \in A$, we must have
\begin{align*}
  \sum_{j \in A} p_{ij} = 1, \qquad i \in A,
\end{align*}
which contradicts the irreducibility of the over all routing matrix $\ol{P}$, and therefore $A = \emptyset$. This proves the first half of (c). It also implies that $\alpha_{j} = \mu_{j}$ for all $j \in \sr{J}$. Hence, if \eq{mu 2} holds with strict inequality, then \eq{traffic 2} implies that
\begin{align*}
  \alpha_{j} = \lambda_{j} + \sum_{k\in \sr{J}} \mu_{k} p_{kj} > \mu_{j}.
\end{align*}
This proves the remaining part of (c).
\end{proof}

We next characterize the conditions in \lem{unstable 1} by the gradient vector of ${\gamma}(\vc{\theta})$ and $\gamma_{s,i}(\vc{\theta})$ at $\vc{\theta} = \vc{0}$. Define the gradient operator $\nabla$ as
\begin{align}
\label{eq:gradient 1}
  \nabla {\gamma}(\vc{\theta}) = \left(\frac {\partial} {\partial \theta_{1}} {\gamma}(\vc{\theta}), \frac {\partial} {\partial \theta_{2}} {\gamma}(\vc{\theta}), \ldots, \frac {\partial} {\partial \theta_{d}} {\gamma}(\vc{\theta})\right).
\end{align}
Since
\begin{align}
\label{eq:gradient es 1}
  \left. \frac {\partial} {\partial \theta_{i}} \gamma_{e,i}(\theta_{i}) \right|_{\theta_{i} = 0} = \lambda_{i}, \qquad \left. \frac {\partial} {\partial \theta_{i}} \gamma_{s,i}(\vc{\theta}) \right|_{\vc{\theta}=\vc{0}} = - \mu_{i}, \qquad \left. \frac {\partial} {\partial \theta_{j}} \gamma_{s,i}(\vc{\theta}) \right|_{\vc{\theta}=\vc{0}} = \mu_{i} p_{ij},
\end{align}
and $\vc{\alpha}^{(0)} = \vc{\lambda} (I - P)^{-1}$, we have
\begin{align}
\label{eq:gradient 2}
  \nabla {\gamma}(\vc{0}) = \vc{\lambda} - \vc{\mu} (I - P), \qquad \nabla {\gamma}(\vc{0}) (I - P)^{-1} = \vc{\alpha}^{(0)} - \vc{\mu} .
\end{align}
Using these facts, we have geometric interpretations for the conditions in \lem{unstable 1} by the curves of ${\gamma}(\vc{\theta}) = 0$ and $\gamma_{s,i}(\vc{\theta}) = 0$ for $i \in \sr{J}$. For this, we introduced vectors $\vc{t}_{i} \in \dd{R}^{d}$ for $i \in \sr{J}$ such that 
\begin{align}
\label{eq:t vector 1}
  \br{\nabla \gamma_{s,j}(\vc{0}), \vc{t}_{i}} = 0, \quad j \ne i, \qquad \br{\nabla \gamma_{s,i}(\vc{0}), \vc{t}_{i}} > 0.
\end{align}
Note that this $\vc{t}_{i}$ is uniquely determined except for its length $\|\vc{t}_{i}\|$.

\begin{lemma}
\label{lem:t 1}
Let $T = (\vc{t}_{1}, \vc{t}_{2}, \ldots, \vc{t}_{d})$, then, for some positive vector $\vc{a}$,
\begin{align}
\label{eq:t 1}
  T = - (I - P)^{-1} \Delta_{\vc{a}},
\end{align}
and therefore $T$ is non-singular and $\vc{t}_{i} \le \vc{0}$ with $t_{ii} < 0$ for all $i \in \sr{J}$.
\end{lemma}
\begin{proof}
Since \eq{gradient es 1} yields
\begin{align}
\label{eq:gamma 4}
  \nabla \gamma_{s,j}(\vc{0}) = \mu_{j}((p_{j1}, p_{j2}, \ldots, p_{jd}) - \vcn{e}_{j}), \qquad j \in \sr{J},
\end{align}
\eq{t 1} is immediate from \eq{t vector 1}.
\end{proof}

\begin{lemma}
\label{lem:geometric 1}
(a) For $j \in \sr{J}$, the condition \eq{mu 1} holds if and only if the $j$-th entry of the gradient vector $\nabla {\gamma}(\vc{0})$ is not positive. (b) For each $k \in \sr{J}$, $\mu_{k} < \alpha^{(0)}_{k}$ if and only if $\br{\nabla {\gamma}(\vc{0}), \vc{t}_{k}} < 0$.
\end{lemma}

\begin{remark}
\label{rem:geometric 1}
$\gamma_{s,i}(\vc{\theta}) = 0 \; (> 0)$ if and only if $q_{i}(\vc{\theta}) = 1\; (> 1, respectively)$ by \eq{gamma 1}, so they present the same geometric curve. However, the gradients $\nabla \gamma_{i}(\vc{\theta})$ and $\nabla q_{i}(\vc{\theta})$ may not be identical. In particular, $\nabla \gamma_{i}(\vc{0}) = \mu_{i} \nabla q_{i}(\vc{0})$.
\end{remark}

\begin{proof}
(a) is immediate from the first equation of \eq{gradient 2}. (b) It follows from \eq{t 1} that 
\begin{align*}
  \vc{\alpha}^{(0)} - \vc{\mu} = \nabla {\gamma}(\vc{0}) (I - P)^{-1} = -  \nabla {\gamma}(\vc{0}) (\vc{t}_{1}, \vc{t}_{2}, \ldots, \vc{t}_{d}) \Delta_{\vc{a}}^{-1},
\end{align*}
Thus, $\mu_{k} < \alpha^{(0)}_{k}$ if and only if $\br{\nabla {\gamma}(\vc{0}), \vc{t}_{k}} < 0$.
\end{proof}

\begin{figure}[h] 
   \centering
   \includegraphics[height=5.2cm]{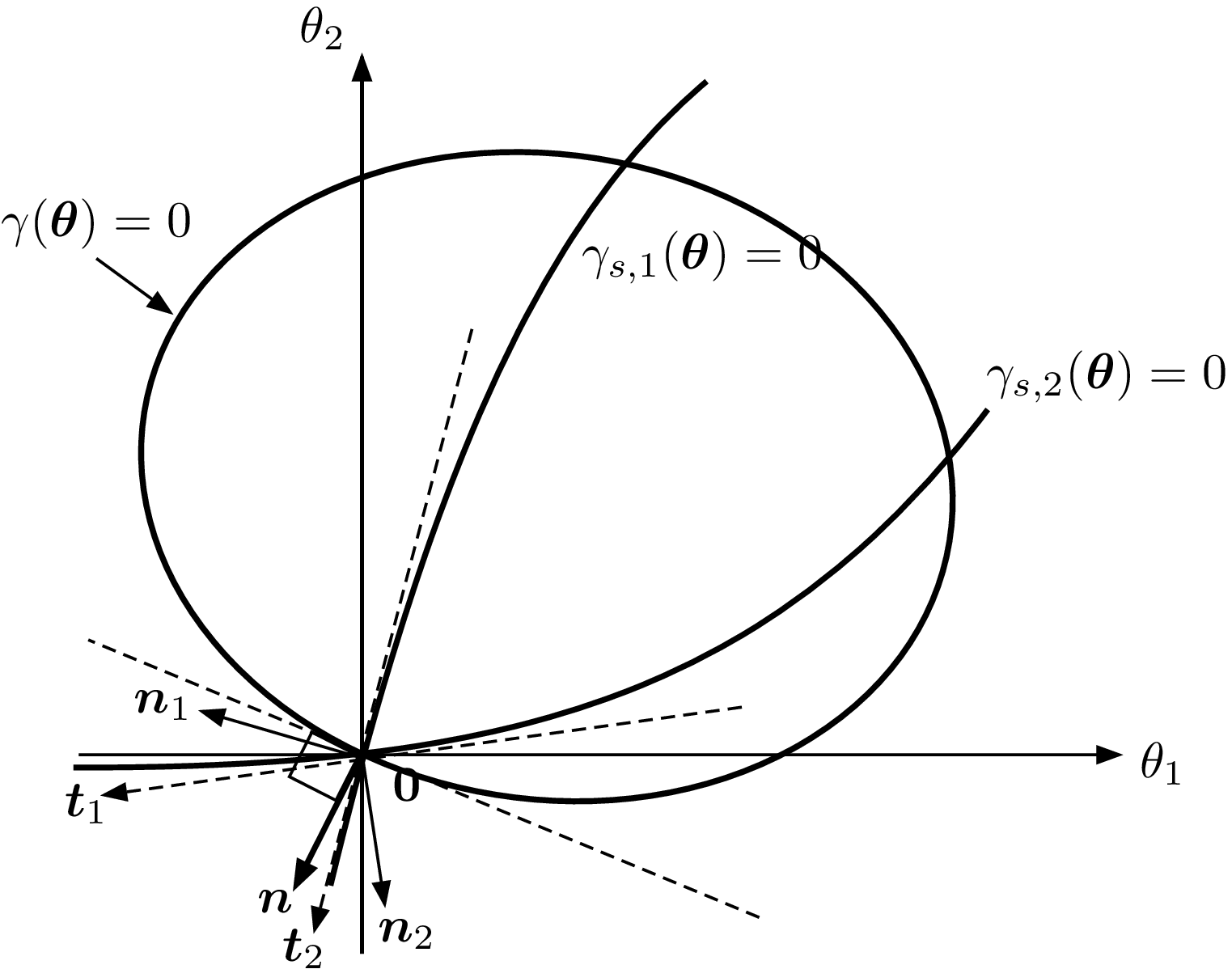} \hspace{2ex} \includegraphics[height=5.2cm]{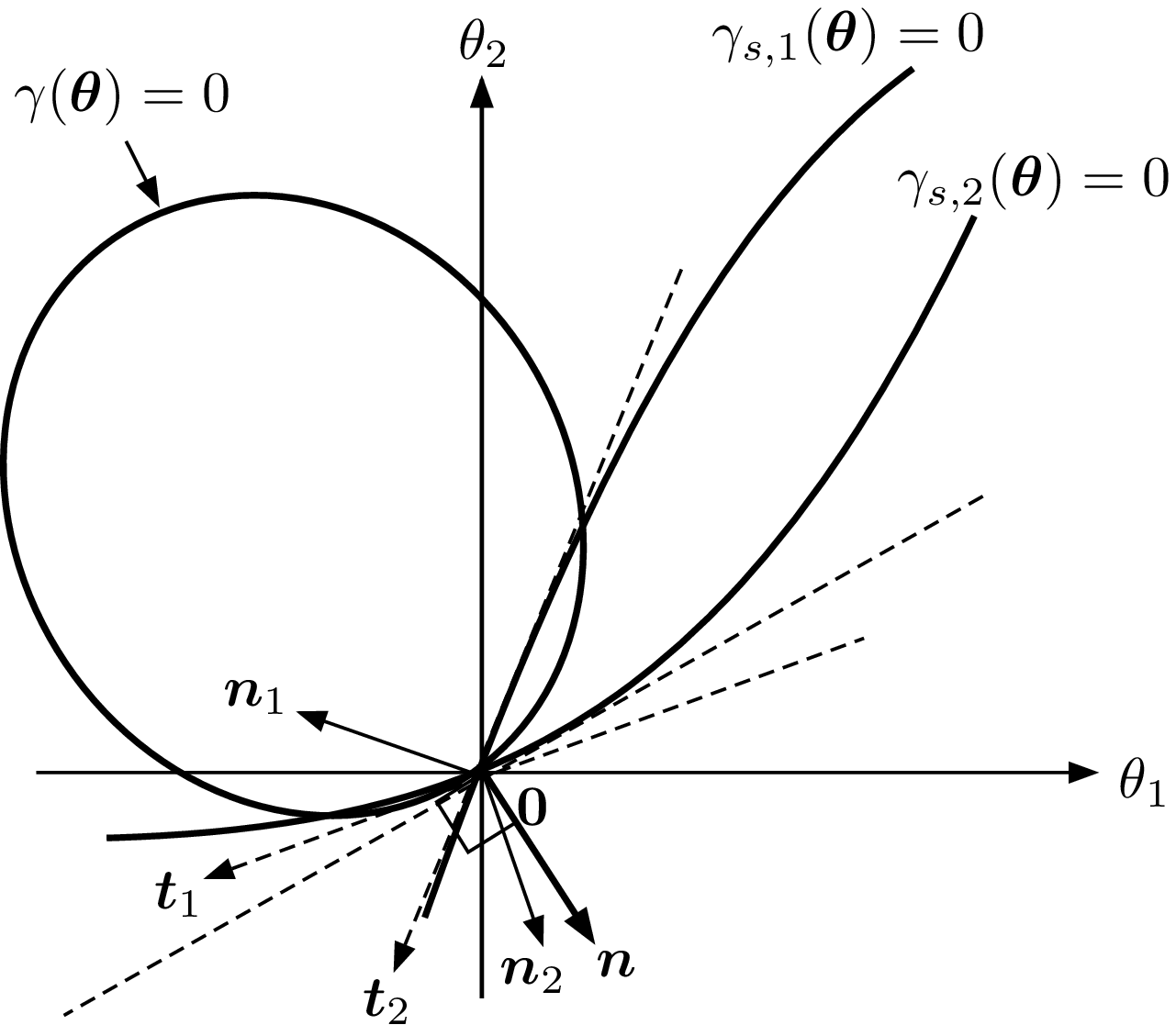}
   \caption{Geometric objects for $d=2$, where $\vc{n} = \nabla {\gamma}(\vc{0})$, $\vc{n}_{i} = \nabla \gamma_{i}(\vc{0})$ for $i=1,2$.}
   \label{fig:2d-GJN-C}
\end{figure}

\fig{2d-GJN-C} illustrates the two cases for $d=2$. The left panel shows that $\vc{n} < \vc{0}$ and  $\br{\vc{n}, \vc{t}_{i}} > 0$ for $i=1,2$, and both stations are stable, which is directly from $\vc{n} < \vc{0}$. The right panel shows that $n_{1} > 0, n_{2} < 0$, $\br{\vc{n}, \vc{t}_{1}} < 0$ and $\br{\vc{n}, \vc{t}_{2}} > 0$, and station $1$ is unstable while station $2$ is stable by Lemmas \lemt{unstable 1} and \lemt{geometric 1}.

\lem{geometric 1} together with Lemmas \lemt{unstable 1} and \lemt{t 1} provides us geometric interpretations of the stability and instability conditions for stations through curves $\gamma(\vc{\theta})=0$ and $\gamma_{s,i}(\vc{\theta})=0$ for $i \in \sr{J}$. We will use them for the GJN before and after change of measure.

\subsection{Tail asymptotics}
\label{sect:tail}

We now return to the assumption that the GJN is stable (see, e.g., the left panel of \fig{2d-GJN-C}). Under this assumption, we will use the following sets for considering the tail asymptotics of the stationary distribution. Let
\begin{align*}
 & \Gamma^{\rs{in}} = \{\vc{\theta} \in \dd{R}^{d}; {\gamma}(\vc{\theta}) < 0\}, \qquad \overleftarrow{\Gamma}^{\rs{in}} = \{\vc{\theta} \in \dd{R}^{d}; \vc{\theta} < \vc{\theta}', \exists \vc{\theta}' \in \Gamma^{\rs{in}} \}, \\
 & \Gamma^{\rs{out}} = \{\vc{\theta} \in \dd{R}^{d}; \gamma(\vc{\theta}) > 0\},\qquad \overrightarrow{\Gamma}^{\rs{out}} = \{\vc{\theta} \in \dd{R}^{d}; \{\vc{\theta}' \in \dd{R}^{d}; \vc{\theta} < \vc{\theta}'\} \cap \Gamma^{\rs{in}} = \emptyset \},
\end{align*}
For $A \subset \sr{J}$, let $\Gamma_{A} = \{\vc{\theta} \in \dd{R}^{d}; \gamma_{s,i}(\vc{\theta}) > 0, \forall i \in \sr{J} \setminus A\}$, and let
\begin{align*}
 & \Gamma^{\rs{in}}_{A} = \Gamma^{\rs{in}} \cap \Gamma_{A}, \qquad \overleftarrow{\Gamma}^{\rs{in}}_{A} = \{\vc{\theta} \in \overleftarrow{\Gamma}^{\rs{in}}; \vc{\theta} < \vc{\theta}', \exists \vc{\theta}' \in \Gamma^{\rs{in}}_{A} \}.
\end{align*}
In particular, for $A = \{k\}$ with $k \in \sr{J}$ and $u = \rs{in}$ or $\rs{out}$, $\Gamma^{u}_{A}$ is simply denoted by $\Gamma^{u}_{k}$. Those sets are open and connected sets. We denote their boundaries by putting operator $\partial$ like $\partial \Gamma^{\rs{in}}$, which is $\{\vc{\theta} \in \dd{R}^{d}; {\gamma}(\vc{\theta}) = 0\}$. Obviously, $\partial \Gamma^{\rs{in}} = \partial \Gamma^{\rs{out}}$, and $\partial \overleftarrow{\Gamma}^{\rs{in}} = \partial \overrightarrow{\Gamma}^{\rs{out}}$.

Note that $\Gamma^{\rs{in}}$ is a non-empty bounded and convex set because $\gamma(\vc{\theta})$ is convex and diverges as $\|\vc{\theta}\|$ goes to infinity in any direction, and therefore $\Gamma^{\rs{out}}$ is also not empty. We check below that $\Gamma^{\rs{in}}_{A}$ is not empty for $A \ne \emptyset$.

\begin{lemma}
\label{lem:geometric 2}
Assume that the GJN is stable, and let $A \subset \sr{J}$. (a) If $A \ne \emptyset$, then $\Gamma^{\rs{in}}_{A}$ is not empty, and contains some $\vc{\theta} \ge \vc{0}$ with $\theta_{i} > 0$ for all $i \in A$. (b) Define
\begin{align*}
 \Gamma^{\rs{cx}}_{A} = \Big\{\vc{\theta} \in \Gamma^{\rs{in}}; \sum_{i \in \sr{E}} \gamma_{e,i}(\vc{\theta}) + \sum_{j \in A} \gamma_{s,j}(\vc{\theta}) < 0\Big\},
\end{align*}
then $\Gamma^{\rs{cx}}_{A}$ is convex, $\Gamma^{\rs{in}}_{A} \subset \Gamma^{\rs{cx}}_{A}$, and $\partial \Gamma^{\rs{in}} \cap \Gamma_{A} \subset \partial \Gamma^{\rs{cx}}_{A}$.
\end{lemma}
\begin{proof}
  (a) We note two facts. Firstly, $\br{\nabla {\gamma}(\vc{0}), \vc{t}_{i}} > 0$ for all $i \in \sr{J}$ by \lem{geometric 1} and the stability condition \eq{stability 1}. Secondly, $\vc{t}_{i} \le 0$ with $t_{ii} < 0$ by \lem {t 1}. These facts imply that $b_{i}(-\vc{t}_{i}) \in \Gamma$ for some $b_{i} > 0$. Let $H^{\rs{in}}_{i} = \{\vc{x} \in \Gamma; \br{\nabla \gamma_{s,i}(\vc{0}),\vc{x}} \ge 0\}$. Since $H^{\rs{in}}_{i}$ is a convex set, $H^{\rs{in}}_{\sr{J} \setminus A} \equiv \cap_{i \in \sr{J} \setminus A} H^{\rs{in}}_{i}$ is also convex, and obviously contains $b_{j}(-\vc{t}_{j})$ for $j \in A$. Hence, their convex combination is also in $H^{\rs{in}}_{\sr{J} \setminus A}$, and nonnegative with positive entries for $j \in A$ because $b_{j}(-\vc{t}_{j}) \ge 0$ and $b_{j}(-t_{jj}) > 0$ for all $j \in A$. Furthermore, $H^{\rs{in}}_{\sr{J} \setminus A} \subset \Gamma^{\rs{in}}_{A}$ because $\vc{x} \in H^{\rs{in}}_{i}$ implies that $\gamma_{s,i}(\vc{x}) > 0$ for $\vc{x} \ne \vc{0}$. Thus, (a) is proved. \\
(b) Since $\gamma_{e,i}$ and $\gamma_{s,j}$ are convex functions, $\Gamma^{\rs{cx}}_{A}$ is a convex set. Since $\gamma_{s,i}(\vc{\theta}) > 0$ for all $i \in \sr{J} \setminus A$ for $\vc{\theta} \in \Gamma^{\rs{in}}_{A}$, we have, for $\vc{\theta} \in \Gamma^{\rs{in}}$,
\begin{align*}
  \sum_{i \in \sr{E}} \gamma_{e,i}(\vc{\theta}) + \sum_{j \in A} \gamma_{s,j}(\vc{\theta}) = \gamma(\vc{\theta}) - \sum_{i \in \sr{J} \setminus A} \gamma_{s,i}(\vc{\theta}) < 0,
\end{align*}
which proves that $\Gamma^{\rs{in}}_{A} \subset \Gamma^{\rs{cx}}_{A}$. If $\vc{\theta} \in \partial \Gamma^{\rs{in}} \cap \Gamma_{A}$, then $\gamma(\vc{\theta}) = 0$ and $\gamma_{s,i}(\vc{\theta}) > 0$ for all $i \in \sr{J} \setminus A$, and therefore $\vc{\theta} \in \partial \Gamma^{\rs{cx}}_{A}$.
\end{proof}

\begin{remark}
\label{rem:geometric 2}
Since $\gamma_{s,i}(\vc{\theta}) = 0$ is equivalent to $q_{i}(\vc{\theta}) = 1$, $\vc{\theta} \in \partial \Gamma_{A}$ if and only if
\begin{align*}
  \sum_{j \in \sr{J} \setminus A} (\delta_{ij} - p_{ij}) e^{\theta_{j}} = \sum_{j \in A} p_{ij} e^{\theta_{j}} + p_{i0}, \qquad i \in \sr{J} \setminus A,
\end{align*}
where $\delta_{ij} = 1(i=j)$. Since $d - |A|$ dimensional matrix $P^{(\sr{J} \setminus A)} \equiv \{p_{ij}; i, j \in \sr{J} \setminus A\}$ is strictly substochastic, $I - P^{(\sr{J} \setminus A)}$ is invertible, and its inverse is nonnegative. Hence, if $\theta_{i} > 0$ for $i \in A$ and $\vc{\theta} \in \partial \Gamma_{A}$, then $\theta_{j} \ge 0$ for $j \in \sr{J} \setminus A$ since $p_{ij}$ may vanish for $i \in A$.
\end{remark}

We now present main results, which are proved in \sectn{theorem}. For this, we use the following notations. For $\vc{x} \in \dd{R}^{d}$ and $A \subset \sr{J}$, let $\vc{x}_{A}$ be the $|A|$ dimensional vector which is obtained from $\vc{x}$, dropping its $i$-entry of $\vc{x}$ for all $i \in \sr{J} \setminus A$. Let
\begin{align*}
  & \varphi_{k}(\vc{\theta}) = \dd{E}(e^{\br{\vc{\theta},\vc{L}}} 1(L_{k} = 0)), \qquad k \in \sr{J}, \qquad \varphi(\vc{\theta}) = \dd{E}(e^{\br{\vc{\theta},\vc{L}}}),\\
 & r_{*}(\vcn{e}_{k}) = \sup\big\{ \theta_{k}; \vc{\theta} \in \Gamma^{\rs{in}}_{k} \cap M_{k}, \varphi_{k}(\vc{\theta}) < \infty\big\}, \qquad k \in \sr{J},
\end{align*}
where $M_{k} = \{\vc{\theta} \in \dd{R}^{d}; \theta_{i} \ge 0, \forall i \in \sr{J} \setminus \{k\}\} \cup \{\vc{\theta} \in \dd{R}^{d}; \theta_{i} < 0, \forall i \in \sr{J} \setminus \{k\}\}$.
Note that $M_{k} = \dd{R}^{2}$ for $d=2$, and therefore $r_{*}(\vcn{e}_{k}) = r_{\{k\}}(\vcn{e}_{k})$. For $A \subset \sr{J}$ and $\vc{c} \in \overrightarrow{U}_{d}$, that is, unit direction vector $\vc{c}$, let 
\begin{align*}
 & r_{A}(\vc{c}) = \sup\big\{ \br{\vc{\theta}, \vc{c}}; \vc{\theta} \in \Gamma^{\rs{in}}_{A}, \varphi_{i}(\vc{\theta}) < \infty, \forall i \in A \big\},\\
 & m_{A}(\vc{c}) = \sup\big\{ u; u \vc{c} \in \overleftarrow{\Gamma}^{\rs{in}}_{A}, \varphi_{i}(\vc{\theta}) < \infty, \forall i \in A \big\}.
\end{align*}
Note that $r_{A}(\vc{c}) \le m_{A}(\vc{c})$ because $\|\vc{c}\| = 1$.

\begin{theorem} 
\label{thr:decay 1}
Assume that the GJN is stable, and let $B_{0}$ be a compact subset of $\dd{R}_{+}^{d}$. (a) For $k \in \sr{J}$,
\begin{align}
\label{eq:upper 1a}
 & \limsup_{x \to \infty} \frac 1x \log \dd{P}(\vc{L} \in x \vcn{e}_{k} + B_{0}) \le - r_{*}(\vcn{e}_{k}).
\end{align}
(b) If the uniformly bounded assumption (A1) is satisfied, then, for $\vc{c} \in \overrightarrow{U}_{d}$,
\begin{align}
\label{eq:upper boundary 1}
 & \limsup_{x \to \infty} \frac 1x \log \dd{P}(\vc{L} \in x \vc{c} + B_{0}) \le - \max\{ r_{A}(\vc{c}); \vc{c}_{A} > \vc{0}_{A}, A \in 2^{\sr{J}} \setminus \emptyset\},\\
\label{eq:upper marginal 1}
 & \limsup_{x \to \infty} \frac 1x \log \dd{P}(\br{\vc{c},\vc{L}} >  x) \le - \max\{ m_{A}(\vc{c}); A \in 2^{\sr{J}} \setminus \emptyset\}.
\end{align}
\end{theorem}

For $B \subset \dd{R}_{+}^{d}$, define a convex corn as
\begin{align*}
  {\rm Corn}(B) = \{\vc{x} \in \dd{R}_{+}^{d}; u \vc{x} \in B, \exists u > 0\}.
\end{align*}

\begin{figure}[h] 
   \centering
   \includegraphics[height=4.1cm]{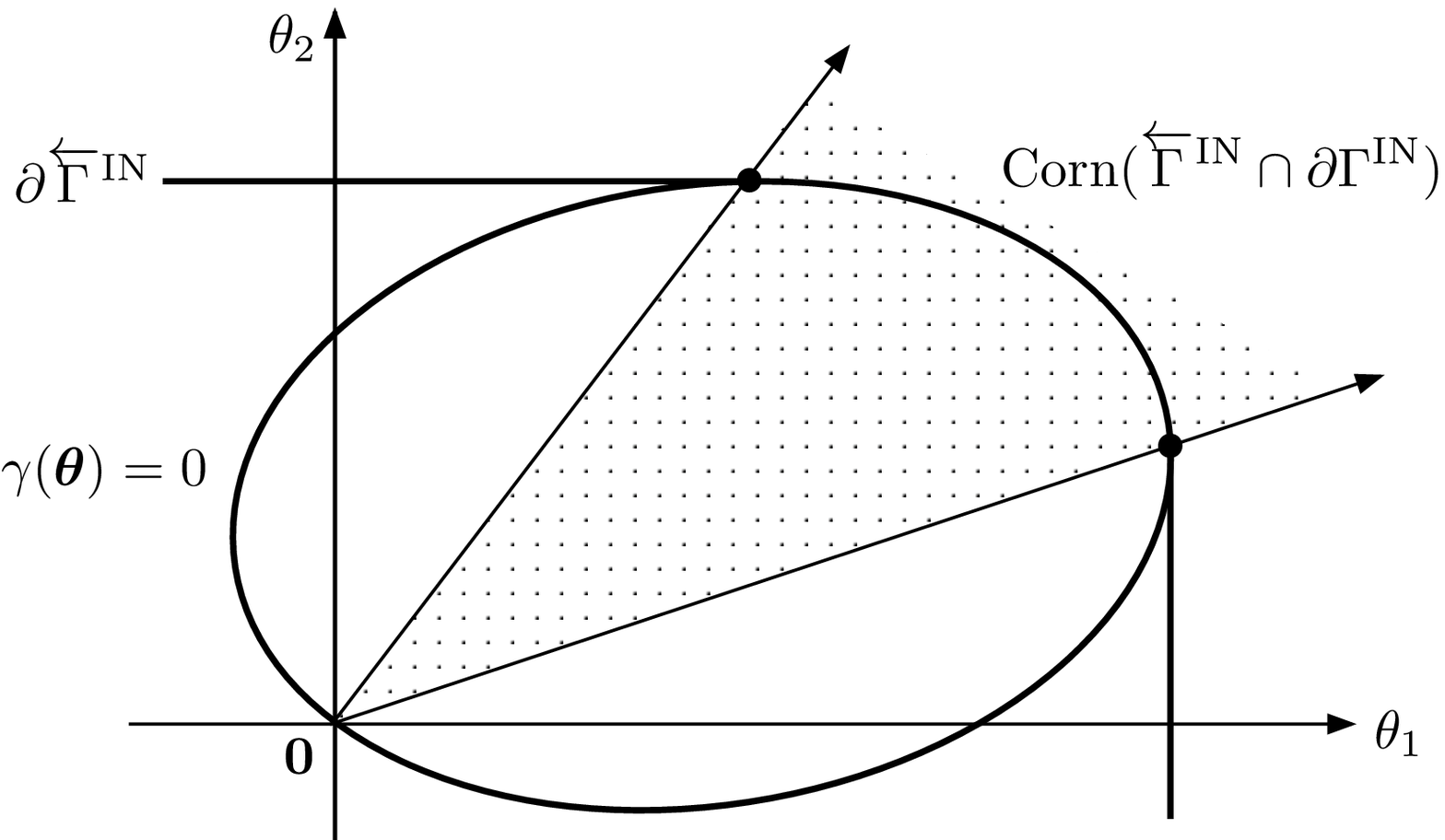} \hspace{1ex} \includegraphics[height=4.1cm]{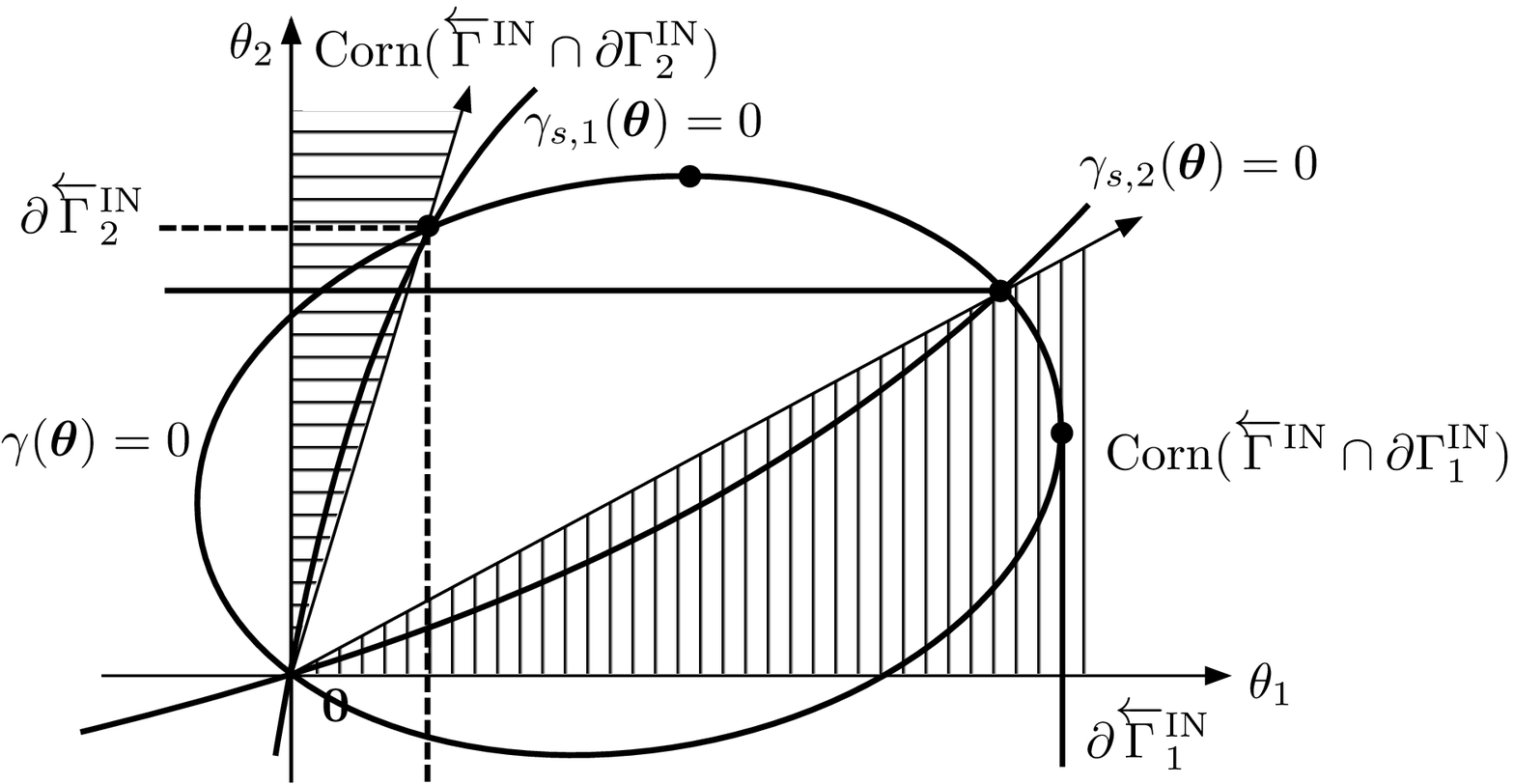} 
   \caption{Corns used in \thr{decay 2}  below}
   \label{fig:convex corn}
\end{figure}

\begin{theorem} 
\label{thr:decay 2}
Assume that the GJN is stable. (a) For $d=2$, let $B_{0}$ be a compact set of $\dd{R}_{+}^{2}$, then, for $k=1,2$,
\begin{align}
\label{eq:lower 1a}
 & \liminf_{x \to \infty} \frac 1x \log \dd{P}(\vc{L} \in x \vcn{e}_{k} + B_{0}) \ge - r_{*}(\vcn{e}_{k}).
\end{align}
(b) For general $d \ge 2$ and $\vc{c} \in \overrightarrow{U}_{d}$ if $\vc{c} \in {\rm Corn}(\overleftarrow{\Gamma}^{\rs{in}} \cap \partial \Gamma^{\rs{in}})$,
\begin{align}
\label{eq:lower 2}
 & \liminf_{x \to \infty} \frac 1x \log \dd{P}(\br{\vc{c},\vc{L}} > x) \ge - \sup\big\{u \ge 0; u \vc{c} \in \Gamma^{\rs{in}} \big\}.
\end{align}
(c) For $d = 2$, $k=1,2$ and $\vc{c} \in \overrightarrow{U}_{2}$, if $\vc{c} \in {\rm Corn}(\overleftarrow{\Gamma}^{\rs{in}} \cap \partial \Gamma^{\rs{in}}_{k})$,
\begin{align}
\label{eq:lower 1b}
 & \liminf_{x \to \infty} \frac 1x \log \dd{P}(\br{\vc{c},\vc{L}} > x ) \ge - \sup\{u \ge 0; u \vc{c} \in \overleftarrow{\Gamma}^{\rs{in}}_{k}\}.
\end{align}
\end{theorem}

For $d=2$, we can get bounds explicitly. For this, let
\begin{align}
\label{eq:cp 1a}
 & \delta_{1} = \sup\{ \theta_{1} \ge 0; \vc{\theta} \in \Gamma^{\rs{in}}_{1}, \theta_{2} \le \delta_{2} \},\\
\label{eq:cp 1b}
 & \delta_{2} = \sup\{ \theta_{2} \ge 0; \vc{\theta} \in \Gamma^{\rs{in}}_{2}, \theta_{1} \le \delta_{1} \},
\end{align}
which are known to have a unique solution $\vc{\delta} = (\delta_{1}, \delta_{2})$ (see the proof of \cor{decay 1} in \sectn{theorem}), and define
\begin{align*}
  \sr{D}_{2} = \{\vc{\theta} \in \overleftarrow{\Gamma}^{\rs{in}}; \theta_{i} < \delta_{i}, i=1,2\}.
\end{align*}
Then we have the following corollary.

\begin{corollary}
\label{cor:decay 1}
Assume the stable GJN has two stations ($d=2$). (a) For $k = 1,2$,
\begin{align}
\label{eq:decay 1}
 & \lim_{x \to \infty} \frac 1x \log \dd{P}(\vc{L} \in x \vcn{e}_{k} + B_{0}) = - \sup\big\{ \theta_{k}; \vc{\theta} \in \sr{D}_{2} \big\}.
\end{align}
(b) If (A1) is satisfied, then, for $\vc{c} \in \overrightarrow{U}_{2}$,
\begin{align}
\label{eq:upper 3}
 & \limsup_{x \to \infty} \frac 1x \log \dd{P}(\vc{L} \in x \vc{c} + B_{0}) \le - \sup\big\{ \br{\vc{\theta}, \vc{c}}; \vc{\theta} \in \sr{D}_{2} \big\},\\
\label{eq:decay 2}
 & \lim_{x \to \infty} \frac 1x \log \dd{P}(\br{\vc{c}, \vc{L}} > x) = - \sup\big\{ u; u \vc{c} \in \sr{D}_{2} \big\}.
\end{align}
\end{corollary}

It is notable that $\sr{D}_{2}$ have been obtained as the convergence domain of $\varphi(\vc{\theta}) \equiv \dd{E}(e^{\br{\vc{\theta},\vc{L}}})$, and used to derive \eq{decay 2} for the two station JGN with phase type $F_{e,i}, F_{s,j}$ in Theorem 4.2 of \cite{Miya2014}. The asymptotic \eq{decay 1} in the coordinate directions is not derived in \cite{Miya2014}, but can be obtained from Theorem 3.2 there. We here have asymptotic \eq{decay 1} without the phase type assumption. We conjecture that the assumption (A1) can be removed from all the results, but it seems a quite hard problem.

Similar results to \eq{decay 1} and \eq{decay 2} are known for a reflecting random walks on the quarter plane (e.g., see \cite{Miya2009,Miya2011}) and semi-martingale reflecting two dimensional Brownian motions, SRBM for short (see \cite{DaiMiya2011}). On the other hand, the asymptotic \eq{upper 3} is new for the GJN, but known for the two dimensional SRBM (\cite{AvraDaiHase2001,DaiMiya2013}), where \eq{upper 3} is sharpened.

For $d \ge 3$, there is very little known about the tail asymptotics of the stationary distribution not only for the GJN but also a reflecting random walk and SRBM.  There are some studies in the framework of sample path large deviations, but those results need to solve certain optimization problems, which are hard to solve even numerically (e.g., see \cite{Maje2009}). Contrary to them, \eq{upper boundary 1} and \eq{upper marginal 1} may be used to get explicit bounds, using ideas for a reflecting random walk (see Theorem 6.1 of \cite{Miya2011}).

\section{Change of measure for GJN}
\label{sect:change}
\setnewcounter

In this section, we present some preliminary results for proving Theorems \thrt{decay 1} and \thrt{decay 2} and \cor{decay 1}. A change of measure is typically used in the theory of large deviations. We also use it, and construct a new measure using a multiplicative functional, which is obtained from the martingale $M_{\vc{J}(v),\vc{\theta}}(\cdot)$ in \sectn{terminal}. However, we assume $\vc{J}(v) = \emptyset$ in this section for making arguments simpler. It also suffices for major applications in the later sections.

Thus, the new measure is constructed from $M_{\vc{\theta}}(\cdot) \equiv M_{\emptyset,\vc{\theta}}(\cdot)$. For this, we first drive a multiplicative functional. Its derivation is rather standard, but will be detailed because it is crucial for our arguments. Our major interest in this section is to see how the GJN is modified under the new measure. It is important for us to specifically identify its modeling parameters, which has not been studied in the literature except for the single queue case (see \cite{Miya2017}), and may have an independent interest.

\subsection{Multiplicative functional}
\label{sect:multiplicative}

Let $Y(t)$ be a left-continuous process, which is called predictable because $Y(t-)$ is $\sr{F}_{t-}$-measurable. Assume that $Y(\cdot)$ has bounded in each finite interval.  Recall that $M_{0}(\cdot)$, $M(\cdot)$ and $A(\cdot)$ be defined by \eq{martingale 2}, \eq{martingale 3} and \eq{BV 1}, respectively. Assume that the terminal condition \eq{terminal 1} is satisfied. Assume that $M(\cdot)$ is an $\sr{F}_{t}$-martingale under $\dd{P}_{\vc{x}}$ for each $\vc{x} \in S$.

We define the integral of $Y(\cdot)$ with respect to martingale $M(\cdot)$ by
\begin{align*}
  Y\cdot M(t) \equiv 1 + \int_{0}^{t} Y(u) dM(u),
\end{align*}
where integration is a natural extension of a Riemann-Stieltjes integral (see Section 4d of Chapter I of \cite{JacoShir2003}). For a positive valued test function $f$, choose $Y(t)$ as
\begin{align*}
  Y(t) = \frac 1{f(X(0))} \exp\Big( - \int_{0}^{t} \frac {\sr{A}f(X(u))} {f(X(u))} du \Big),
\end{align*}
which is obviously positive and continuous in $t$ and adapted to $\sr{F}_{t}$. Hence, $Y\cdot M(\cdot)$ is martingale. We denote it by $E^{f}(\cdot)$. Thus, it follows from \eq{martingale 3} that
\begin{align}
\label{eq:martingale 5}
  E^{f}(t) & = 1 + \int_{0}^{t} Y(u) \Big( df(X(u)) - \sr{A}f(X(u)) du \Big)  \nonumber\\
  & = 1 + \int_{0}^{t} Y(u) df(X(u)) + \int_{0}^{t} f(X(u)) Y'(u) du \nonumber\\
  & = 1 + [Y(u) f(X(u))]_{0}^{t}  =  \frac {f(X(t))}{f(X(0))} \exp\Big( - \int_{0}^{t} \frac {\sr{A}f(X(u))} {f(X(u))} du, \Big) ,
\end{align}
which is an $\sr{F}_{t}$-martingale under $\dd{P}_{\vc{x}}$.

On the other hand, $E^{f}(\cdot)$ is a multiplicative functional because it is right-continuous, $E^{f}(0)=1$, $\dd{E}(E^{f}(t)) = 1$ and
\begin{align*}
  E^{f}(s+t) = E^{f}(s) \Theta_{s} \circ E^{f}(t), \qquad s,t \ge 0,
\end{align*}
where
\begin{align*}
 \Theta_{s} \circ E^{f}(t) = \frac {f(X(s+t))}{f(X(s))} \exp\Big( - \int_{s}^{s+t} \frac {\sr{A}f(X(u))} {f(X(u))} du, \Big).
\end{align*}
Thus, we can define a probability measure $\widetilde{\dd{P}}^{f}_{\vc{x}}$ for an initial state $\vc{x} \in S$ by
\begin{align}
\label{eq:change 1}
  \frac {d\widetilde{\dd{P}}_{\vc{x}}^{f}} {d\dd{P}_{\vc{x}}} \Big|_{\sr{F}_{t}} = E^{f}(t), \qquad t \ge 0,
\end{align}
because $E^{f}(\cdot)$ is a martingale (see \cite{KuniWata1963} for details). We refer to \eq{change 1} as exponential change of measure. Let $\dd{P}_{\nu}$ and $\widetilde{\dd{P}}_{\nu}^{f}$ be probability measures such that $\dd{P}_{\nu}(C) = \int_{S} \dd{P}_{\vc{x}}(C)\nu(d\vc{x})$ and $\widetilde{\dd{P}}_{\nu}^{f}(C) = \int_{S} \widetilde{\dd{P}}^{f}_{\vc{x}}(C)\nu(d\vc{x})$ for $X(0)$ to have a probability distribution $\nu$ on $S$, \eq{change 1} implies that, for a non-negative $\sr{F}_{t}$-measurable random variable $U$ with finite expectation, we have
\begin{align}
\label{eq:change 2}
   \widetilde{\dd{E}}_{\nu}^{f}(U) = \dd{E}_{\nu}(UE^{f}(t)), \qquad \dd{E}_{\nu}(U) = \widetilde{\dd{E}}_{\nu}^{f}(E^{f}(t)^{-1}U),
\end{align}
where $\dd{E}_{\nu}$ and $\widetilde{\dd{E}}_{\nu}^{f}$ represent the expectations concerning $\dd{P}_{\nu}$ and $\widetilde{\dd{P}}_{\nu}^{f}$, respectively. Similarly, for conditional expectations, we have, for $0 \le s < t$,
\begin{align}
\label{eq:change 3}
  \widetilde{\dd{E}}(U|\sr{F}_{s}) = \dd{E}\Big(U\frac {E^{f}(t)}{E^{f}(s)} \Big|\sr{F}_{s}\Big), \qquad \dd{E}(U |\sr{F}_{s}) = \widetilde{\dd{E}}\Big(U\frac {E^{f}(s)}{E^{f}(t)} \Big|\sr{F}_{s}\Big).
\end{align}
One can easily check this equation from the definition of a conditional expectation (see, e.g., Section III.3 of \cite{JacoShir2003}).

When $f = f_{\vc{\theta}}$ of \eq{test 1} and $M = M_{\vc{\theta}}$ of \eq{martingale 4} with $\vc{J}(v) = \emptyset$, we denote  denotes $\widetilde{\dd{P}}^{f_{\vc{\theta}}}_{\vc{x}}$ by $\widetilde{\dd{P}}^{(\vc{\theta)}}_{\vc{x}}$. If $\vc{J}(v) \ne \emptyset$, then the new measure is denoted by $\widetilde{\dd{P}}^{(\vc{J}(v),\vc{\theta)}}_{\vc{x}}$.

\subsection{GJN under the new measure}
\label{sect:new}

Let us consider how the GJN is modified under the new measure $\widetilde{\dd{P}}^{(\vc{\theta})}_{\vc{x}}$. A general principle for change of measure is considered for a PDMP in \cite{PalmRols2002}, but we need to compute specific modeling parameters. For this, we follow the method of \cite{Miya2017} studied for a single queue with many heterogeneous servers. We here modify it for the GJN. Since the differential operator $\sr{A}$ is unchanged because it works on a deterministic part of the sample path of $X(\cdot)$, we only need to consider the jump kernel $\sr{K}$. Denote it under $\widetilde{\dd{P}}^{(\vc{\theta})}_{\vc{x}}$ by $\widetilde{\sr{K}}^{(\vc{\theta})}$.

We first write $E^{f_{\vc{\theta}}}(t)$ explicitly as
\begin{align}
\label{eq:Ef 1}
  & E^{f_{\vc{\theta}}}(t) = e^{\br{\vc{\theta}, \vc{L}(t) - \vc{L}(0)} - w(\vc{\theta}, \vc{R}(t) - \vc{R}(0))- {\gamma}(\vc{\theta}) t + \int_{0}^{t} \sum_{i \in \sr{J}} \gamma_{s,i}(\vc{\theta}) 1(L_{i}(u) = 0) du },
\end{align}
where $w(\vc{\theta},\vc{y}) = w_{\emptyset}(\vc{\theta},\vc{y})$ (see \eq{w J 1}), that is,
\begin{align*}
  w(\vc{\theta}, \vc{y}) & = \sum_{i \in \sr{E}} \gamma_{e,i}(\theta_{i}) y_{e,i} + \sum_{i \in \sr{J}} \gamma_{s,i}(\theta_{i}) y_{s,i}.
\end{align*}

Our first task is to compute the distributions of $T_{e,i}, T_{s,j}$ under $\widetilde{\dd{P}}^{(\vc{\theta})}_{\vc{x}}$. These distributions (moment generating functions) are denoted, respectively, by $F_{e,i}^{(\vc{\theta})}$ ($\widehat{F}_{e,i}^{(\vc{\theta})}$) and $F_{s,j}^{(\vc{\theta})}$ ($\widehat{F}_{s,j}^{(\vc{\theta})}$). Recall $\beta^{(v)}_{F}$ of \eq{beta 1}, and denote $\beta_{F_{e,i}}^{(\infty)}$ and $\beta_{F_{s,j}}^{(\infty)}$ simply by $\beta_{e,i}^{(\infty)}$ and $\beta_{s,j}^{(\infty)}$, respectively. Similar to Lemma 4.4 of \cite{Miya2017}, we have

\begin{lemma}
\label{lem:new T 1}
For each $\vc{\theta} \in \dd{R}^{d}$, $v \in (0,\infty]$ and $\eta \in \dd{R}$,
\begin{align}
\label{eq:new T 1a}
 & \widehat{F}_{e,i}^{(\vc{\theta})}(\eta) = 
 e^{\theta_{i}} \widehat{F}_{e,i}(\eta - \gamma_{e,i}(\theta_{i})), \qquad {\eta} \le \beta_{e,i}^{(\infty)} + \gamma_{e,i}(\theta_{i}), i \in \sr{E} \setminus J_{e}(v), \\
\label{eq:new T 2a}
 & \widehat{F}_{s,i}^{(\vc{\theta})}(\eta) = 
  q_{i}(\vc{\theta}) \widehat{F}_{s,i}(\eta - \gamma_{s,i}(\vc{\theta})), \qquad {\eta} \le \beta_{s,i}^{(\infty)} + \gamma_{s,i}(\vc{\theta}), i \in \sr{J} \setminus J_{s}(v). 
\end{align}
\end{lemma}

Since $\widehat{F}_{e,i}^{(\vc{\theta})}(0) = \widehat{F}_{s,i}^{(\vc{\theta})}(0) = 1$ by \eq{xi 3}, \eq{light tail 1} and \eq{gamma 1}, $F_{e,i}^{(\vc{\theta})}$ and $F_{s,j}^{(\vc{\theta})}$ are proper distribution functions. Let
\begin{align*}
  \lambda^{(\vc{\theta})}_{i} = (\widetilde{\dd{E}}_{e,i}^{(\vc{\theta})}(T_{e,i}))^{-1}, \qquad \mu^{(\vc{\theta})}_{i} = (\widetilde{\dd{E}}^{(\vc{\theta})}_{s,i}(T_{s,i}))^{-1},
\end{align*}
where $\widetilde{\dd{E}}^{(\vc{\theta})}_{e,i}$ and $\widetilde{\dd{E}}^{(\vc{\theta})}_{s,i}$ represent the conditional expectations under $\widetilde{\dd{E}}^{(\vc{\theta})}_{\vc{x}}$ just before time when external arrivals and service completion, respectively, at station $i$ occur. Then, by \lem{new T 1}, we have
\begin{align}
\label{eq:new mean 1}
 & \lambda^{(\vc{\theta})}_{i} =
  (e^{\theta_{i}} \widehat{F}_{e,i}'(- \gamma_{e,i}(\theta_{i})))^{-1}, \qquad i \in \sr{E} \\
\label{eq:new mean 2}
 & \mu^{(\vc{\theta})}_{i} =
  (q_{i}(\vc{\theta}) \widehat{F}_{s,i}'(- \gamma_{s,i}(\theta_{i})))^{-1}, \qquad  i \in \sr{J} .
\end{align}

The jump kernel $\sr{K}$ is changed to $\widetilde{\sr{K}}^{(\vc{\theta})}$ as
\begin{align}
\label{eq:new K 1}
  \widetilde{\sr{K}}^{(\vc{\theta})} & 1_{B_{\ell} \times B_{e} \times B_{s}}(\vc{x}) \nonumber\\
  & = \left\{
\begin{array}{ll}
 \widetilde{\dd{P}}_{e,i}^{(\vc{\theta})}(\vc{z}+\vc{e}_{i} \in B_{\ell}, \vc{y}_{e} + T_{e,i} \vc{e}_{i} \in B_{e}, \vc{y}_{s} \in B_{s}), & y_{e,i} = 0,  \\
 \widetilde{\dd{P}}_{s,i}^{(\vc{\theta})}(\vc{z}-\vc{e}_{i} + \vc{e}_{j} \in B_{\ell}, \vc{y}_{e} \in B_{e}, \vc{y}_{s} + T_{s,i} \vc{e}_{i} \in B_{s}) , & y_{s,i} = 0, 
\end{array}
\right.
\end{align}
where $\theta_{0} = 0$. Hence, the routing probability from station $i$ to $j$ under $\widetilde{\dd{P}}^{(\vc{\theta})}_{\vc{x}}$ is
\begin{align}
\label{eq:pij 1}
  p^{(\vc{\theta})}_{ij} \equiv e^{-\theta_{i} + \theta_{j}} p_{ij}/q_{i}(\vc{\theta}).
\end{align}

Thus, the GJN (generalized Jackson network) keeps the same network structure under the new probability measure $\widetilde{\dd{P}}^{(\vc{\theta})}_{\vc{x}}$, but their modeling primitives, $F_{e,i}$, $F_{s,j}$ and $p_{ij}$ are changed to $F_{e,i}^{(\vc{\theta})}$, $F_{s,j}^{(\vc{\theta})}$ and $p^{(\vc{\theta})}_{ij}$, respectively, which do not depend on the initial state $\vc{x}$. Let
\begin{align*}
  q_{i}^{(\vc{\theta})}(\vc{\eta}) = e^{-\eta_{i}} \Bigg(\sum_{j \in \sr{J}} p^{(\vc{\theta})}_{ij} e^{\eta_{j}} + p^{(\vc{\theta})}_{i0}\Bigg), \qquad i \in \sr{J},
\end{align*}
which is $q_{i}(\vc{\eta})$ under $\widetilde{\dd{P}}^{(\vc{\theta})}_{\vc{x}}$, where $\vc{\eta} \in \dd{R}^{d}$ is a variable. From this definition and \eq{pij 1}, we have
\begin{align*}
  q^{(\vc{\theta})}_{i}(\vc{\eta}) = \frac {q_{i}(\vc{\eta}+\vc{\theta})}{q_{i}(\vc{\theta})} .
\end{align*}

Similarly to the original network model, we define $\gamma^{(\vc{\theta})}_{e,i}(\eta_{i})$, $\gamma^{(\vc{\theta})}_{s,j}(\vc{\eta})$ as the unique solutions of the following equations.
\begin{align*}
  e^{\eta_{i}} \widehat{F}^{(\vc{\theta})}_{e,i}(-\gamma^{(\vc{\theta})}_{e,i}(\eta_{i})) = 1, \quad i \in \sr{E}, \quad q_{i}^{(\vc{\theta})}(\vc{\eta}) \widehat{F}^{(\vc{\theta})}_{s,i}(-\gamma^{(\vc{\theta})}_{s,i}(\vc{\eta})) = 1, \quad i \in \sr{J},
\end{align*}
for $v \in (0,\infty]$. These definitions yield
\begin{align*}
 & \gamma^{(\vc{\theta})}_{e,i}(\eta_{i}) = \gamma_{e,i}(\eta_{i}+\theta_{i}) - \gamma_{e,i}(\theta_{i}),\\
 & \gamma^{(\vc{\theta})}_{s,j}(\vc{\eta}) = \gamma_{s,j}(\vc{\eta}+\vc{\theta}) - \gamma_{s,j}(\vc{\theta}),
\end{align*}
and define $\gamma^{(\vc{\theta})}(\vc{\eta})$ as
\begin{align*}
  \gamma^{(\vc{\theta})}(\vc{\eta})  = \sum_{i \in \sr{E}} \gamma^{(\vc{\theta})}_{e,i}(\eta_{i}) + \sum_{i \in \sr{J}} \gamma^{(\vc{\theta})}_{s,i}(\vc{\eta}) = {\gamma}(\vc{\eta}+\vc{\theta}) - {\gamma}(\vc{\theta}).
\end{align*}

We immediately see from these formulas that
\begin{align}
\label{eq:new gradient 1}
  \nabla \gamma^{(\vc{\theta})}_{s,i}(0) = \nabla \gamma_{s,i}(\vc{\theta}), \qquad \nabla \gamma^{(\vc{\theta})}(0) = \nabla {\gamma}(\vc{\theta}).
\end{align}
Similarly to \eq{gradient 2} and \eq{gamma 3}, we have
\begin{align}
\label{eq:new gradient 2}
 & \nabla \gamma^{(\vc{\theta})}(\vc{0}) = \vc{\lambda}^{(\vc{\theta})} - \vc{\mu}^{(\vc{\theta})} (I - P^{(\vc{\theta})})\\
\label{eq:new gradient 3}
 &  \nabla \gamma^{(\vc{\theta})}_{s,j}(\vc{0}) = \mu^{(\vc{\theta})}_{j}((p^{(\vc{\theta})}_{j1}, p^{(\vc{\theta})}_{j2}, \ldots, p^{(\vc{\theta})}_{jd}) - \vcn{e}_{j}), \qquad j \in \sr{J}.
\end{align}
Hence, we can update Lemmas \lemt{unstable 1} and \lemt{geometric 1} in the exactly same way for the network model under $\widetilde{\dd{P}}^{(\vc{\theta})}_{\vc{x}}$.

The following lemma is almost immediate from \eq{pij 1} and \eq{new gradient 3}, but will be useful to check the conditions in \lem{geometric 1}. Similar to $\vc{t}_{i}$ of \eq{t vector 1}, we define $\vc{t}^{(\vc{\theta})}_{i} \in \dd{R}^{d}$ by
\begin{align*}
  \br{\nabla \gamma_{s,j}(\vc{\theta}), \vc{t}^{(\vc{\theta})}_{i}} = 0, \quad j \ne i, \qquad \br{\nabla \gamma_{s,i}(\vc{\theta}), \vc{t}^{(\vc{\theta})}_{i}} > 0.
\end{align*}
Hence, similar to \lem{t 1}, we have the following lemma.

\begin{lemma}
\label{lem:geometric sign 1}
Let $T^{(\vc{\theta})}$ be the matrix whose $i$-th column is $\vc{t}^{(\vc{\theta})}_{i}$, then $T^{(\vc{\theta})}$ is non-singular and not positive, that is, $\vc{t}^{(\vc{\theta})}_{i} \le \vc{0}$ with $t^{(\vc{\theta})}_{ii} < 0$ for all $i \in \sr{J}$. 
\end{lemma}

\section{Proofs}
\label{sect:proofs}
\setnewcounter

The goal of this section is to prove the theorems and their corollary. A main idea is to use the new measure introduced in \sect{new} by appropriately choosing the parameter $\vc{\theta}$. Some of its arguments are parallel to those in Section 4 of \cite{Miya2017}, but we require more lemmas because of the state space for the queue lengths is multidimensional. We start to represent the stationary tail probability under the new measure.

\subsection{A procedure for deriving tail asymptotics}
\label{sect:method}

Recall the notation $w(\vc{\theta},\vc{y})$, and, for $\vc{R}(t) = (\vc{R}_{e}(t), \vc{R}_{s}(t))$, 
\begin{align*}
  w(\vc{\theta},\vc{R}(t)) = \br{\vc{\gamma}_{e}(\vc{\theta}), \vc{R}_{e} (t)} + \br{\vc{\gamma}_{s}(\vc{\theta}), \vc{R}_{s}(t)}.
\end{align*}
Then, it follows from \eq{change 2} and \eq{Ef 1} that, for a given initial distribution $\nu$,
\begin{align}
\label{eq:exp 1}
  d\dd{P}_{\nu} & = (E^{f_{\vc{\theta}}}(t))^{-1} d\widetilde{\dd{P}}^{(\vc{\theta})}_{\nu}\nonumber\\
  &= f_{\vc{\theta}}(X(0)) e^{-\br{\vc{\theta}, \vc{L}(t)} + w(\vc{\theta},\vc{R}(t)) + {\gamma}(\vc{\theta}) t - \sum_{i \in \sr{J}} \gamma_{s,i}(\vc{\theta}) \int_{0}^{t} 1(L_{i}(u) = 0) du } d\widetilde{\dd{P}}^{(\vc{\theta})}_{\nu}, \quad \mbox{on } \sr{F}_{t},
\end{align}
where we recall that $f_{\vc{\theta}}(X(0)) = e^{\br{\vc{\theta}, \vc{L}(0)} - w(\vc{\theta},\vc{R}(t))}$.

We take the initial distribution $\nu$ in the following way. Let $S_{1} = \dd{Z}_{+}^{d}$, and let $\tau^{\rs{ex}}_{A}, \tau^{\rs{re}}_{A}$ be the first exit from and return times of $L(t)$ to $\partial_{A} S_{1}$ such that $\tau^{\rs{ex}}_{A} < \tau^{\rs{re}}_{A}$, where
\begin{align*}
  \partial_{A} S_{1} = \cup_{i \in A} \{\vc{z} \in S_{1}; z_{i} = 0\}.
\end{align*}
Let $\nu^{-}_{A}$ the distribution of $X(\tau^{\rs{ex}}_{A}-)$ given that $X(0)$ is subject to the normalized stationary distribution limited on $\partial_{A} S_{1}$. This $\nu^{-}_{A}$ is taken for $\nu$ in \eq{exp 1}. Denote a random vector subject to the stationary distribution of $X(t)$ by $X \equiv (\vc{L}, \vc{R}_{e}, \vc{R}_{s})$. Then, the cycle formula yields, for $x > 0$ and $B(x) \subset S_{1} \setminus \partial_{A} S_{1}$,
\begin{align}
\label{eq:cycle 1}
  \dd{P}(\vc{L} \in B(x)) = b(A) \dd{E}_{\nu^{-}_{A}}\Big( \int_{0}^{\tau^{\rs{re}}_{A}} 1(\vc{L}(u) \in B(x)) du \Big),
\end{align}
where $b(A) = \dd{P}(\vc{L} \in S_{1} \setminus \partial_{A} S_{1}) / \dd{E}_{\nu^{-}_{A}}(\tau^{\rs{re}}_{A} - \tau^{\rs{ex}}_{A})$. We here are interested in the asymptotic of $\dd{P}(\vc{L} \in B(x))$ as $x \to \infty$.

We apply change of measure to \eq{cycle 1}. For this, let $\tau_{x}$ be a stopping time such that
\begin{align}
\label{eq:stopping 1}
  \tau_{x} \le \inf\{t \ge 0; \vc{L}(t) \in B(x)\}, \qquad x > 0,
\end{align}
which is a crucial condition in our approach. Let
\begin{align*}
 & Y(t) = \dd{E}_{\nu^{-}_{A}} \Big(\int_{t}^{\tau^{\rs{re}}_{A}} 1(\vc{L}(u) \in B(x)) du \Big|\sr{F}_{t} \Big), 
\end{align*}
then it follows from \eq{exp 1} with $\nu = \nu^{-}_{A}$ that
\begin{align}
\label{eq:exp 2}
 \dd{P} (\vc{L} \in B(x)) = b(A) & \widetilde{\dd{E}}^{(\vc{\theta})}_{\nu^{-}_{A}} \Bigg[ f_{\vc{\theta}}(X(0)) Y(\tau_{x}) 1(\tau_{x} < \infty) e^{- \br{\vc{\theta}, \vc{L}(\tau_{x})} + w(\vc{\theta},\vc{R}(\tau_{x}))}\nonumber\\
  & \times \exp\Bigg({\gamma}(\vc{\theta}) \tau_{x} - \sum_{i \in \sr{J} \setminus A} \gamma_{s,i}(\vc{\theta}) \int_{0}^{\tau_{x}} 1(L_{i}(u) = 0) du\Bigg) \Bigg].
\end{align}

We are now ready to consider the asymptotic of $\dd{P} (\vc{L} \in B(x))$ as $x \to \infty$. For its upper bound, we take the following steps.

\renewcommand{\labelenumi}{(\roman{enumi})}
\begin{enumerate}
\item [1)] Choose $\vc{\theta} \in \Gamma^{\rs{in}}_{A}$, which implies that
\begin{align*}
  \exp\left( {\gamma}(\vc{\theta}) \tau_{x} - \sum_{i \in \sr{J} \setminus A} \gamma_{s,i}(\vc{\theta}) \int_{0}^{\tau_{x}} 1(L_{i}(u) = 0) du\right) \le 1.
\end{align*}
\item [2)] Verify that there is a constant $C_{1}$ such that, if $\tau_{x} < \infty$, then
\begin{align}
\label{eq:Y 1}
  Y(\tau_{x}) e^{w(\vc{\theta},\vc{R}(\tau_{x}))} < C_{1}.
\end{align}
\item [3)] Verify that $\widetilde{\dd{E}}^{(\vc{\theta})}_{\nu^{-}_{A}} (f_{\vc{\theta}}(X(0))$ is finite if $\varphi_{i}(\vc{\theta}) < \infty$ for all $i \in A$.
\item [4)] Find finite real-valued functions $\ol{a}_{0}(\vc{\theta})$ and $\ol{a}_{1}(\vc{\theta}) > 0$ such that
\begin{align}
\label{eq:bound 2}
 \ol{a}_{1}(\vc{\theta})x - \br{\vc{\theta},\vc{L}(\tau_{x})} \le \ol{a}_{0}(\vc{\theta}),
\end{align}
then $e^{- \br{\vc{\theta}, \vc{L}(\tau_{x})}}$ is bounded above by $e^{\ol{a}_{0}(\vc{\theta})-\ol{a}_{1}(\vc{\theta})x}$.
\item [5)] Derive an inequality from \eq{exp 2} using 1)--4), divide both sides of this inequality by $x$, and let $x \to \infty$, then take the infimum of the upper bound on $\vc{\theta}$ for which steps 1)--4) work well.
\end{enumerate}

To derive the lower bounds, we modify \eq{exp 2} by replacing $M_{\vc{\theta}}$ by the martingale $M_{\vc{J}(v), \vc{\theta}}$ of \eq{martingale 4} in \lem{martingale 2} choosing the index set for truncation, $\vc{J}(v) \equiv (J_{e}(v), J_{s}(v))$ for each fixed $\vc{\theta} \in \dd{R}^{d}$, as
\begin{align}
\label{eq:J negative 1}
  J_{e}(v) = \{i \in \sr{E}; \gamma_{e,i}(\theta_{i}) < 0\}, \qquad J_{s}(v) = \{i \in \sr{E}; \gamma_{s,i}(\vc{\theta}) < 0\},
\end{align}
and we choose a sufficiently large $v$ such that $\gamma_{e,i}(v,\theta_{i}) < 0$ for all $i \in J_{e}(v)$ and $\gamma_{s,j}(v,\vc{\theta}) < 0$ for all $j \in J_{s}(v)$, which is possible by \lem{concave 1} and the assumption \eq{light tail 1}. Then, $w_{\vc{J}(v)}(\vc{\theta},\vc{y})$ of \eq{w J 1} is bounded below for all $\vc{y} = (\vc{y}_{e}, \vc{y}_{s}) \in \dd{R}^{\sr{E}} \times \dd{R}^{d}$. Namely,
\begin{align*}
  w_{\vc{J}(v)}(\vc{\theta},\vc{y}) \ge v \Bigg(\sum_{i \in J_{e}(v)} \gamma_{e,i}(v,\theta_{i}) + \sum_{i \in J_{s}(v)} \gamma_{s,i}(v,\vc{\theta}) \Bigg) > -\infty.
\end{align*}

 Then, \eq{exp 2} is changed as
\begin{align}
\label{eq:exp 3}
 \dd{P} & (\vc{L} \in B(x)) = b_{\vc{J}_{v}}(A) \widetilde{\dd{E}}^{(\vc{J}(v),\vc{\theta})}_{\nu^{-}_{A}} \Bigg[ f_{\vc{J}(v),\vc{\theta}}(X(0)) Y(\tau_{x}) 1(\tau_{x} < \infty) e^{- \br{\vc{\theta}, \vc{L}(\tau_{x})} }\nonumber\\
  & \times \exp\Bigg(w_{\vc{J}(v)}(\vc{\theta},\vc{R}(\tau_{x})) + {\gamma}(\vc{\theta}) \tau_{x} - \sum_{i \in \vc{J}_{e}(v)} \int_{0}^{\tau_{x}} \gamma_{e,i}(v,\theta_{i}) 1(R_{e,i}(u) > v) du \nonumber\\
  & -  \sum_{i \in \vc{J}_{s}(v)} \int_{0}^{\tau_{x}} \gamma_{s,i}(v,\vc{\theta}) 1(R_{s,i}(u) > v) du -  \sum_{i \in \sr{J} \setminus A} \int_{0}^{\tau_{x}} \gamma_{J_{s}(v),i}(\vc{\theta}) 1(L_{i}(u) = 0) du\Bigg) \Bigg],
\end{align}
where $b_{\vc{J}_{v}}(A)$ is the normalizing constant and
\begin{align*}
  \gamma_{J_{s}(v),i}(\vc{\theta}) = \gamma_{s,i}(v,\vc{\theta}) 1(i \in J_{s}(v)) + \gamma_{s,i}(\vc{\theta}) 1(i \ne J_{s}(v)).
\end{align*}
Note that the first integration term with minus sign in the exponent of \eq{exp 3} is bounded below by $0$ by the choice of $\vc{J}(v)$. We now take the following steps for the lower bounds.

\begin{enumerate}
\setcounter{enumi}{5}
\item [1')] Choose $\vc{\theta} \in \Gamma^{\rs{out}}$ such that $\gamma_{s,i}(\vc{\theta}) < 0$ for all $i \in \sr{J} \setminus A$, which implies $i \in J_{s}(v)$ and, for sufficiently large $v > 0$, 
\begin{align*}
 \exp\Bigg( {\gamma}(\vc{\theta}) \tau_{x} & - \sum_{i \in \sr{J} \setminus A} \int_{0}^{\tau_{x}} \gamma_{s,i}(v,\vc{\theta}) 1(L_{i}(u) = 0) du\Bigg) \ge 1.
\end{align*}
The lower bounds are only used for \thr{decay 2}. Thus, $A = \sr{J}$ for general $d$ and $A = \{k\}$ for $d = 2$.
\item [2')] Verify that there is a constant $C_{2}$ such that, if $\tau_{x} < \infty$, then
\begin{align}
\label{eq:Y 2}
  Y(\tau_{x}) > C_{2}.
\end{align}
\item [3')] Find finite valued functions $\ul{a}_{0}(\vc{\theta}), \ul{a}_{1}(\vc{\theta})$ such that
\begin{align}
\label{eq:bound 3}
 \ul{a}_{1}(\vc{\theta})x - \br{\vc{\theta},\vc{L}(\tau_{x})} \ge \ul{a}_{0}(\vc{\theta}),
\end{align}
then $e^{- \br{\vc{\theta}, \vc{L}(\tau_{x})}}$ is bounded below by $e^{\ul{a}_{0}(\vc{\theta})-\ul{a}_{1}(\vc{\theta})x}$.
\item [4')] Find a subset $U$ of $\partial_{A} S$ such that
\begin{align}
\label{eq:hitting 1}
 & \liminf_{x \to \infty} \widetilde{\dd{P}}^{(\vc{\theta})}_{\nu^{-}_{A}}(X(0) \in U, \tau_{x} < \infty) > 0,\\
\label{eq:U 1}
 & \widetilde{\dd{E}}^{(\vc{\theta})}_{\nu^{-}_{A}} (1(X(0) \in U) f_{\vc{\theta}}(X(0)) < \infty.
\end{align}
\item [5')] The final step is similar to 5) of the upper bound.
\end{enumerate}

In this procedure, we first need to find appropriate $B(x)$ and $\tau_{x}$ so that \eq{bound 3} and \eq{hitting 1} hold, then go through steps. Among them, \eq{hitting 1} is technically most demanding.  

\subsection{Lemmas for tail asymptotics}
\label{sect:technical}

For an open set or closed $B \subset \dd{R}_{+}^{d}$, we define $\tau^{\rs{in}}_{B}$ as
\begin{align*}
  \tau^{\rs{in}}_{B} \equiv \inf\{t > 0; \vc{L}(t) \in B\}.
\end{align*}
This notation will be used in lemmas below.

\begin{lemma}
\label{lem:Y 1}
For each $A \subset \sr{J}$, $x > 0$ and $B(x) \subset S_{1} \setminus \partial_{A} S_{1}$, let $ \tau_{x} = \tau^{\rs{in}}_{B(x)}$. If there is a positive constant $c_{0}$ to be independent of $x$ such that
\begin{align}
\label{eq:bounded 1}
  \sup\{|\br{\vc{\theta}, (\vc{z} - \vc{z}')}|; \vc{z}, \vc{z}' \in B(x)\} < c_{0}\|\theta\|, \qquad \vc{\theta} \in \dd{R}^{d},
\end{align}
then \eq{Y 1} holds for some $C_{1} > 0$, which is independent of $x$.
\end{lemma}

\begin{proof}
We follow the proving method of Lemma 4.6 of \cite{Miya2017}. We replace $\vc{L}(\cdot)$ by $\vc{H}(\cdot)$ such that $\vc{H}(\cdot)$ is obtained from $\vc{L}(\cdot)$ removing the reflecting boundary $\partial_{A} S_{1}$. Hence, the state space of $\vc{H}(\cdot)$ has no limitation concerning entries with indexes in $A$. For $t > 0$, let
\begin{align*}
  \tau_{0}(t) = \inf\{ u \ge t + \tau^{\rs{in}}_{B(x)}; \vc{H}(u) \in B(x)\},
\end{align*}
then, on $\{\tau^{\rs{in}}_{B(x)} < \infty\}$, 
\begin{align*}
  t & \le \int_{\tau^{\rs{in}}_{B(x)}}^{\tau^{\rs{re}}_{A}} 1(\vc{L}(u) \in B(x)) du \le 
\int_{\tau^{\rs{in}}_{B(x)}}^{\infty} 1(\vc{H}(u) \in B(x)) du
\end{align*}
implies that $t \le \tau_{0}(t) - \tau^{\rs{in}}_{B(x)} < \infty$. Hence, we have, on $\{\tau^{\rs{in}}_{B(x)} < \infty\}$,
\begin{align}
\label{eq:bound 1}
  Y(\tau^{\rs{in}}_{B(x)}) e^{w(\vc{\theta},\vc{R}(\tau^{\rs{in}}_{B(x)}))} & = \dd{E}_{\nu^{-}_{A}} \Big(\int_{\tau^{\rs{in}}_{B(x)}}^{\tau^{\rs{re}}_{A}} 1(\vc{H}(u) \in B(x)) du \, e^{w(\vc{\theta},\vc{R}(\tau^{\rs{in}}_{B(x)}))} \Big|\sr{F}_{\tau^{\rs{in}}_{B(x)}} \Big) \nonumber\\
  & \le \int_{0}^{\infty} \dd{E}_{\nu^{-}_{A}}(e^{w(\vc{\theta},\vc{R}(\tau^{\rs{in}}_{B(x)}))} 1(\tau_{0}(t) < \infty)|\sr{F}_{\tau^{\rs{in}}_{B(x)}}) dt.
\end{align}
We evaluate
\begin{align*}
  \dd{E}_{\nu^{-}_{A}}(e^{w(\vc{\theta},\vc{R}(\tau^{\rs{in}}_{B(x)}))} 1(\tau_{0}(t) < \infty)|\sr{F}_{\tau^{\rs{in}}_{B(x)}})
\end{align*}
using change of measure by $\vc{H}(\cdot)$ similar to $\vc{L}(\cdot)$. Let 
\begin{align*}
  J_{e}(v) = \{i \in \sr{E}; \gamma_{e,i}(\theta_{i}) > 0\}, \qquad J_{s}(v) = \{i \in \sr{J}; \gamma_{s,i}(\vc{\theta}) > 0\},
\end{align*}
and we choose a sufficiently large $v$ such that $\gamma_{e,i}(v,\theta_{i}) > 0$ for all $i \in J_{e}(v)$ and $\gamma_{s,j}(v,\vc{\theta}) > 0$ for all $j \in J_{s}(v)$, which is possible by the same reason as used for \eq{J negative 1}.

For change of measure, we use the test function $f_{\vc{J}(v),\vc{\theta}}$ of \eq{test 2} and the martingale $M_{\vc{J}(v),\vc{\theta}}$ of \eq{martingale 4}, where $\vc{L}(t)$ is replaced by $\vc{H}(t)$. Then, the exponential martingale $E^{f}(\cdot)$ is obtained as
\begin{align}
\label{eq:Ef 2}
 E^{f_{\vc{J}(v),\vc{\theta}}}(t) & = \frac {f_{\vc{J}(v),\vc{\theta}}(Y(t))}{f_{\vc{J}(v),\vc{\theta}}(Y(0))} e^{- \gamma_{\vc{J}(v)}(\vc{\theta}) t + \sum_{i \in \sr{J} \setminus A} \gamma_{s,i}(\vc{\theta}) \int_{0}^{t} 1(H_{i}(u) = 0) du } \nonumber\\
 & \qquad \times \exp\Big(\sum_{i \in J_{e}(v)} \gamma_{e,i}(v,\theta_{i}) \int_{0}^{t} 1(R_{e,i}(u) > v) du \nonumber\\
 & \hspace{15ex} + \sum_{i \in J_{s}(v)} \gamma_{s,i}(v,\theta_{i}) \int_{0}^{t} 1(R_{s,i}(u) > v) du \Big),
\end{align}
and define the new measure $\widetilde{\dd{P}}^{(\vc{J}(v),\vc{\theta})}_{\nu}$. Since $\gamma(\vc{\theta}) < 0$ for $\vc{\theta} \in \Gamma^{\rs{in}}_{A}$, there is a sufficient large $v$ such that $\gamma_{\vc{J}(v)}(\vc{\theta}) < 0$. We choose this $v$, then it follows from its conditional expectation version \eq{change 3} that
\begin{align}
\label{eq:tau 0}
  & \dd{E}_{\nu^{-}_{A}} (e^{w_{\vc{J}(v)}(\vc{\theta},\vc{R}(\tau^{\rs{in}}_{B(x)}))} 1(\tau_{0}(t) < \infty)|\sr{F}_{\tau^{\rs{in}}_{B(x)}}) \nonumber\\
 & = \widetilde{\dd{E}}^{(\vc{J}(v),\vc{\theta})}_{\nu^{-}_{A}} \Big(\frac{f_{\vc{J}(v),\vc{\theta}}(Y(\tau^{\rs{in}}_{B(x)}))} {f_{\vc{J}(v),\vc{\theta}}(Y(\tau_{0}(t))} e^{w_{\vc{J}(v)}(\vc{\theta},\vc{R}(\tau^{\rs{in}}_{B(x)}))+\gamma_{\vc{J}(v)}(\vc{\theta}) (\tau_{0}(t) - \tau^{\rs{in}}_{B(x)})} 1(\tau_{0}(t) < \infty) \nonumber\\
 & \qquad \times \exp\Big( - \sum_{i \in J_{e}(v)} \gamma_{e,i}(v,\theta_{i}) \int_{\tau^{\rs{in}}_{B(x)}}^{\tau_{0}(t)} 1(R_{e,i}(u) > v) du \nonumber\\
 & - \sum_{i \in J_{s}(v)} \gamma_{s,i}(v,\theta_{i}) \int_{\tau^{\rs{in}}_{B(x)}}^{\tau_{0}(t)} 1(R_{s,i}(u) > v) du - \sum_{i \in \sr{J} \setminus A} \gamma_{s,i}(\vc{\theta}) \int_{\tau^{\rs{in}}_{B(x)}}^{\tau_{0}(t)} 1(H_{i}(u) = 0) du \Big) \Big|\sr{F}_{\tau^{\rs{in}}_{B(x)}} \Big) \nonumber\\
 & \le \widetilde{\dd{E}}^{(\vc{J}(v),\vc{\theta})}_{\nu^{-}_{A}} \big(e^{-\br{\vc{\theta}, \vc{H}(\tau_{0}(t))} + \br{\vc{\theta}, \vc{H}(\tau^{\rs{in}}_{B(x)})}} \nonumber\\
 & \hspace{20ex} \times e^{w_{\vc{J}(v)}(\vc{\theta},\vc{R}(\tau_{0}(t))) + \gamma_{\vc{J}(v)}(\vc{\theta}) t} 1(\tau_{0}(t) < \infty) \big|\sr{F}_{\tau^{\rs{in}}_{B(x)}} \big) \nonumber\\
 & \le \widetilde{\dd{E}}^{(\vc{J}(v),\vc{\theta})}_{\nu^{-}_{A}} \big(e^{-\br{\vc{\theta}, \vc{H}(\tau_{0}(t))} + \br{\vc{\theta}, \vc{H}(\tau^{\rs{in}}_{B(x)})}}) e^{\sum_{i \in J_{e}(v)} \gamma_{e,i}(v,\theta_{i}) v + \sum_{i \in J_{s}(v)} \gamma_{s,i}(v,\theta_{i}) v + \gamma_{\vc{J}(v)}(\vc{\theta}) t},
\end{align}
since $\tau_{0}(t) - \tau^{\rs{in}}_{B(x)}) \ge t$ on $\{\tau_{0}(t) < \infty\}$ and non-truncated $\gamma_{e,i}(\vc{\theta})$ and $\gamma_{s,j}(\vc{\theta})$ are not positive. We here note that the condition \eq{bounded 1} implies that
\begin{align}
\label{eq:H 1}
  - \br{\vc{\theta}, \vc{H}(\tau_{0}(t))} + \br{\vc{\theta}, \vc{H}(\tau^{\rs{in}}_{B(x)})} \le \|\vc{\theta}\| c_{0},
\end{align}
from which the last term in \eq{tau 0} is bounded by
\begin{align*}  
 & c(v) e^{\gamma_{\vc{J}(v)}(\vc{\theta}) t}, \qquad t \ge 0, v > 0,
\end{align*}
where $c(v) = e^{c_{0} \|\vc{\theta}\| + \sum_{i \in J_{e}(v)} \gamma_{e,i}(v,\theta_{i}) v + \sum_{i \in J_{s}(v)} \gamma_{s,i}(v,\theta_{i}) v}$. Hence, the last term in \eq{bound 1} is bounded by
\begin{align*}
  \int_{0}^{\infty} c(v) e^{\gamma_{\vc{J}(v)}(\vc{\theta}) t} dt = - \frac 1{\gamma_{\vc{J}(v)}(\vc{\theta})} c(v) < \infty.
\end{align*}
This proves the lemma.
\end{proof}

In the proof of \lem{Y 1}, the condition \eq{bounded 1} may be weekend as long as \eq{H 1} holds. However, we also require the conditions \eq{stopping 1} and \eq{bound 2} for $\tau_{x} = \tau^{\rs{in}}_{B(x)}$ to get an upper bound. In the view of these conditions, \eq{bounded 1} is close to be necessary.

\begin{lemma}
\label{lem:finite 1}
We have that $\widetilde{\dd{E}}^{(\vc{\theta})}_{\nu^{-}_{A}} (f_{\vc{\theta}}(X(0)) < \infty$ for $A = \{k\}$ for each $k \in \sr{J}$ if $\varphi_{k}(\vc{\theta}) < \infty$ and $\vc{\theta} \in M_{k}$.
\end{lemma}

\begin{proof}
Since $\dd{P}_{\nu^{-}_{k}}$ is identical with $\widetilde{\dd{P}}^{(\vc{\theta})}_{\nu^{-}_{k}}$ on $\sr{F}_{0}$, it is enough to show that
\begin{align}
\label{eq:finite 2}
  \dd{E}_{\nu^{-}_{k}} & (f_{\vc{\theta}}(X(0))) = \dd{E}_{\nu^{-}_{k}}(e^{\br{\vc{\theta}, \vc{L}(0)} - \br{\vc{\gamma}_{e}(\vc{\theta}), \vc{R}_{e}(0)  } - \br{\vc{\gamma}_{s}(\vc{\theta}), \vc{R}_{s}(0)}}) < \infty.
\end{align}
We first show that $\varphi_{k}(\vc{\theta}) < \infty$ implies that $\dd{E}_{\nu^{-}_{k}}(e^{\br{\vc{\theta}, \vc{L}}}) < \infty$. To see this, let $N_{e,i}(\cdot)$ be the counting process for the exogenous arrivals at station $i$, and let $N_{d,j,i}(\cdot)$ be the counting process for the customers who are completed service at station $j$ and routed to station $i$, then the Palm formulas for stationary point processes yields
\begin{align}
\label{eq:finite 3}
  \dd{E}_{\nu^{-}_{k}}(e^{\br{\vc{\theta}, \vc{L}}}) \le & 1(k \in \sr{E}) \lambda_{k} \dd{E}_{\nu} \Bigg( \int_{0}^{1} e^{\br{\vc{\theta}, \vc{L}(t-)}} 1(\vc{L}(t-) \in \partial_{k} S_{1}) N_{e,k}(dt)  \Bigg) \nonumber\\
  & + \sum_{j \in \sr{J} \setminus \{k\}} \alpha_{j} p_{jk} \dd{E}_{\nu} \Bigg( \int_{0}^{1} e^{\br{\vc{\theta}, \vc{L}(t-)}} 1(\vc{L}(t-) \in \partial_{k} S_{1}) N_{d,j,k}(dt) \Bigg),
\end{align}
where $\partial_{k} S_{1} = \partial_{\{k\}} S_{1}$ and the inequality becomes equality if point processes $N_{e,k}, N_{d,j,k}$ have no common point. Let $N_{e,i} \equiv 0$ for $i \in \sr{J} \setminus \sr{E}$ for convenience, then
\begin{align*}
  \max_{t \in (0,1]} (L_{i}(t) - L_{i}(0)) \le d+1+ N_{e,i}((t_{e,i}(1),1]) + \sum_{j \in \sr{J}} N_{s,j,i}((t_{s,j}(1),1]), \qquad i \in \sr{J},
\end{align*}
where $N_{s,j,i}(\cdot)$ is the stationary counting process for the number of service completions routed to station $i$ when the server at station $j$ is always busy. Since $\vc{L}(0)$ is independent of $N_{e,i}((t_{e,i}(1),1])$ and $N_{s,j}((t_{s,j}(1),1])$ for $i \in \sr{E}, j \in \sr{J}$, \eq{finite 3} implies that
\begin{align*}
  \dd{E}_{\nu^{-}_{k}}(e^{\br{\vc{\theta}, \vc{L}}}) \le \varphi_{k}(\vc{\theta}) \Big(\lambda_{k} + \sum_{j' \in \sr{J} \setminus \{k\}} \alpha_{j'} p_{j'k} \Big) \dd{E}_{\nu}(e^{\sum_{i \in \sr{E}} |\theta_{i}| (d+1+N_{e,i}(0,1] + \sum_{j \in \sr{J}} N_{s,j,i}((0,1])}).
\end{align*}
This proves the claim that $\dd{E}_{\nu^{-}_{k}}(e^{\br{\vc{\theta}, \vc{L}}}) < \infty$ since $N_{e,i}((0,1])$ and $N_{s,j}((0,1])$ for $i \in \sr{E}, j \in \sr{J}$ are independent and have super exponential distributions, that is, their tails are asymptotically faster than any exponential function (see, e.g., Lemma 4.1 of \cite{Miya2017}).

We now prove \eq{finite 2}. Note that its terms multiplied by $\gamma_{e,i}(\theta_{i}) \ge 0$, which is equivalent to $\theta_{i} \ge 0$, or $\gamma_{s,i}(\vc{\theta})$ for $i \in \sr{J} \setminus \{k\}$ can be dropped to bound the second expectation term  in \eq{finite 2} because $\vc{\theta} \in \Gamma^{\rs{in}}_{k}$. Furthermore, $R_{s,k}(0) = T_{s,k}$ under the distribution $\nu^{-}_{k}$. Let $K_{-}(k,\vc{\theta}) = \{i \in \sr{E} \setminus \{k\}; \theta_{i} < 0\}$. Thus, it follows from the equation in \eq{finite 2} that
\begin{align}
  \dd{E}_{\nu^{-}_{k}} & (f_{\vc{\theta}}(X(0))) \le \dd{E}_{\nu^{-}_{k}}\big(e^{\sum_{i \in \sr{J} \setminus \{k\}}\theta_{i} L_{i}(0) - \sum_{i \in K_{-}(k,\vc{\theta})} \gamma_{e,i}(\theta_{i}) R_{e,i}(0)}\big) q_{k}(\vc{\theta})^{-1}.
\end{align}
Thus, \eq{finite 2} is immediate if $\theta_{i} \ge 0$, equivalently, $\gamma_{e,i}(\theta_{i}) \ge 0$, for all $i \in \sr{J} \setminus \{k\}$ since $\theta_{k} \ge 0$ and $\varphi_{k}(\vc{\theta}) < \infty$. Hence, it remains to prove \eq{finite 2} when $\theta_{i} < 0$ for all $i \in \sr{J} \setminus \{k\}$. In this case, \eq{finite 2} is obtained from that
\begin{align}
\label{eq:finite 4}
  \dd{E}_{\nu^{-}_{k}}\big(e^{- W_{e,K_{-}(k,\vc{\theta})}(\vc{\theta},t)}\big) < \infty,
\end{align}
where $W_{e,K_{-}(k,\vc{\theta})}(\vc{\theta},t) = \sum_{i \in K_{-}(k,\vc{\theta})} \gamma_{e,i}(\theta_{i}) R_{e,i}(t)$. Let $\dd{E}_{e,i}$ and $\dd{E}_{d,j}$ represent the expectations concerning the stationary embedded distributions at exogenous arrivals at station $i$ and at departure instants at station $j$, respectively. Then, they are known as Palm distributions (e.g., see \cite{BaccBrem2003}), and obtained as
\begin{align*}
  & \dd{E}_{e,i}\big(e^{- \sum_{i \in K_{-}(A,\vc{\theta})} \gamma_{e,i}(\theta_{i}) R_{e,i}(0)}\big) = \lambda_{i} \dd{E}_{\nu} \Bigg( \int_{0}^{1} e^{- W_{e,K_{-}(A,\vc{\theta})}(\vc{\theta},t)} N_{e,i}(dt)  \Bigg), \nonumber\\
  & \dd{E}_{d,j}\big(e^{- \sum_{i \in K_{-}(A,\vc{\theta})} \gamma_{e,i}(\theta_{i}) R_{e,i}(0)}\big) = \alpha_{j} \dd{E}_{\nu} \Bigg( \int_{0}^{1} e^{- W_{e,K_{-}(A,\vc{\theta})}(\vc{\theta},t)} N_{d,j}(dt) \Bigg),
\end{align*}
where $N_{d,j}(t) = \sum_{i \in \sr{J}} N_{d.j.i}(t)$. From a similar bound in \eq{finite 3}, \eq{finite 4} is obtained from
\begin{align}
\label{eq:finite 5}
  & \dd{E}_{e,i}\big(e^{- \sum_{i \in K_{-}(A,\vc{\theta})} \gamma_{e,i}(\theta_{i}) R_{e,i}(0)}\big) < \infty, \qquad i \in A \cap \sr{E},\\
\label{eq:finite 6}
  & \dd{E}_{d,j}\big(e^{- \sum_{i \in K_{-}(A,\vc{\theta})} \gamma_{e,i}(\theta_{i}) R_{e,i}(0)}\big) < \infty, \qquad j \in \sr{J}.
\end{align}
Since $-\gamma_{e,i}(\theta_{i}) > 0$ for $i \in K_{-}(A,\vc{\theta})$, we can apply Lemma 4.8 of \cite{Miya2017}, which is originally from Lemma 4.2 of \cite{SadoSzpa1995}, and obtain \eq{finite 5} and \eq{finite 6}.
\end{proof}

\begin{figure}[h] 
   \centering
   \includegraphics[height=5cm]{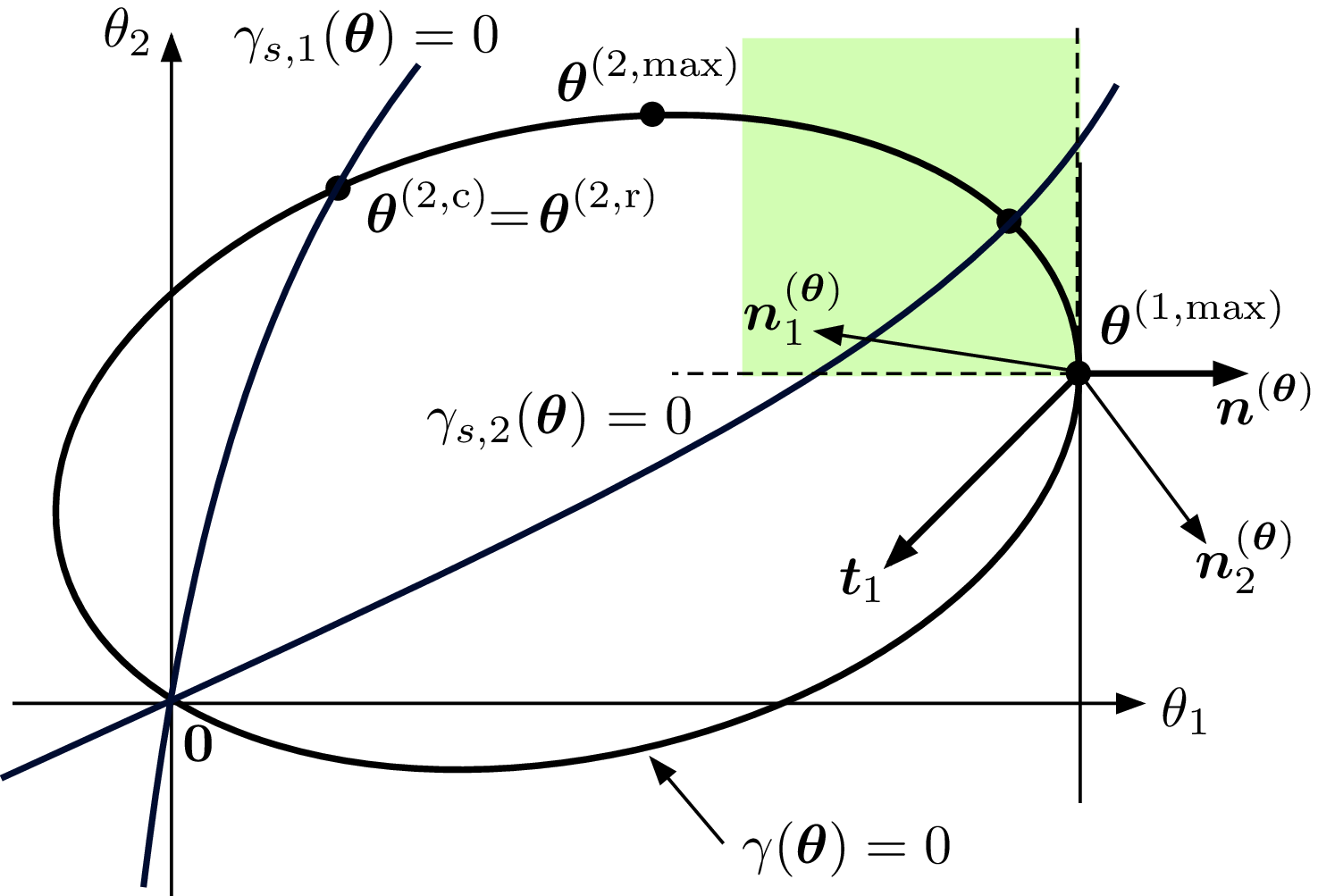} \hspace{1ex} \includegraphics[height=5cm]{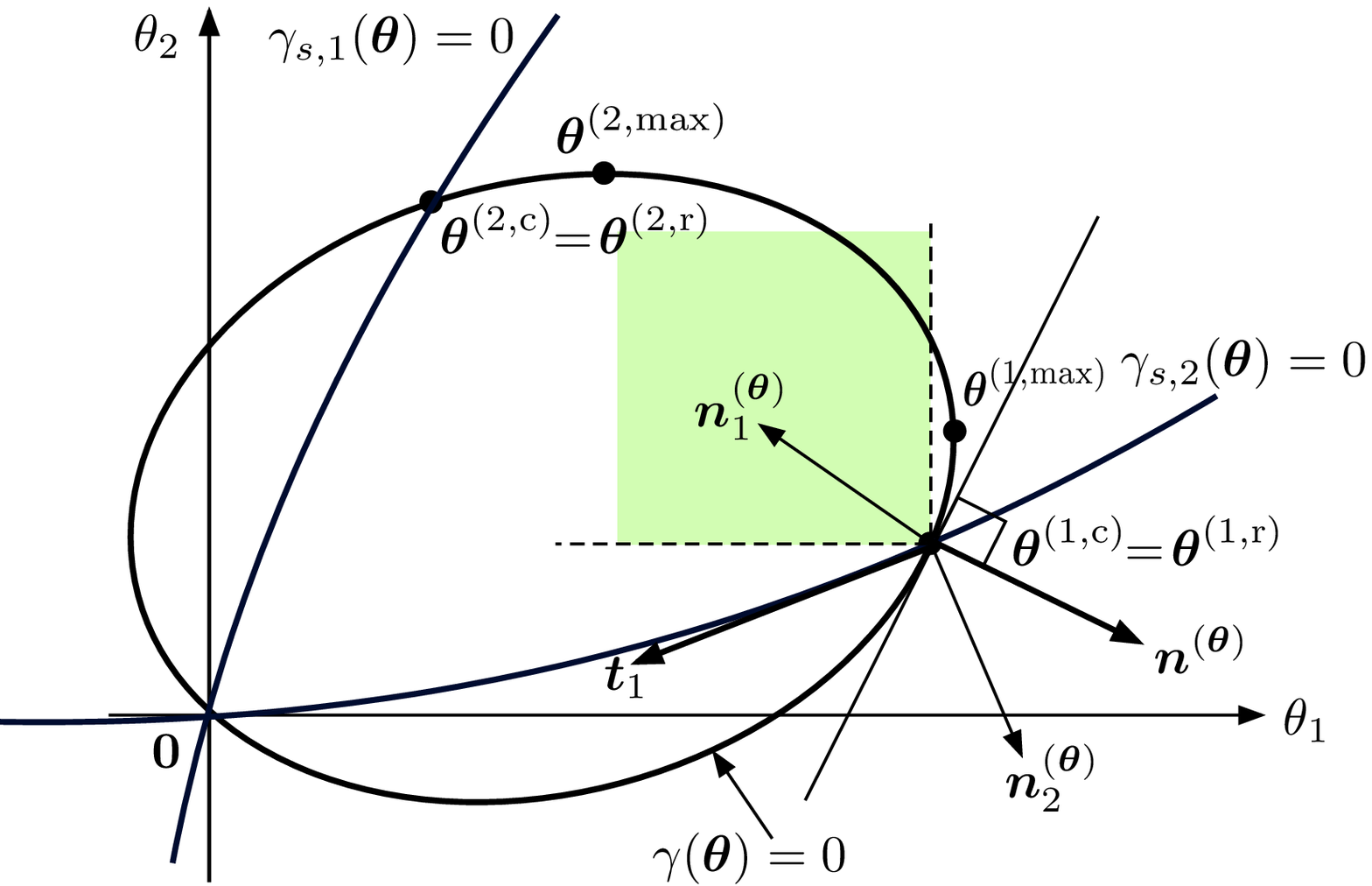}
   \caption{Geometric objects for $d=2$, where $\vc{n}^{(\vc{\theta})} = \nabla {\gamma}(\vc{\theta})$, $\vc{n}^{(\vc{\theta})}_{i} = \nabla \gamma_{i}(\vc{\theta})$ for $i=1,2$.}
   \label{fig:2d-GJN-C3}
\end{figure}

\begin{lemma}
\label{lem:hitting 1}
For $d = 2$, $k =1,2$ and compact set $B_{0} \subset  \dd{R}_{+}^{2}$, let $B_{k}(x) = x\vcn{e}_{k} + B_{0}$, and let
\begin{align}
\label{eq:cp 1}
  \vc{\theta}^{(\rs{cp}_{k})} = \argsup_{\vc{\theta} \in \Gamma^{\rs{in}}_{k}} \theta_{k}, \qquad k=1,2,
\end{align}
then $\liminf_{x \to \infty} \widetilde{\dd{P}}^{(\vc{\theta})}_{\nu^{-}_{\{k\}}}(\tau^{\rs{in}}_{B_{k}(x)} < \infty) > 0$ if $\|\vc{\theta} - \vc{\theta}^{(\rs{cp}_{k})}\|$ is sufficiently small and $\theta_{2} < \theta^{(\rs{cp}_{1})}_{2}$.
\end{lemma}
\begin{proof}
For notational symmetry, we only consider the case for $k=1$. Clearly, the lemma is obtained if $\liminf_{x \to \infty} \widetilde{\dd{P}}^{(\vc{\theta})}_{\nu^{-}_{\{k\}}}(\vc{L}(0) = \vc{z}, \tau^{\rs{in}}_{B_{k}(x)} < \infty) > 0$. It is not hard to see that this is obtained if station $1$ is unstable and station $2$ is stable under $\widetilde{\dd{P}}^{(\vc{\theta})}_{\nu^{-}_{\{k\}}}(\cdot|\vc{L}(0) = \vc{z})$. By \lem{geometric 1} and \eq{new gradient 1}, this is obtained if
\begin{align*}
  [\nabla \gamma(\vc{\theta})]_{2} < 0, \qquad \br{\nabla \gamma(\vc{\theta}), \vc{t}^{(\vc{\theta})}_{1}} < 0,
\end{align*}
which are satisfied if $\vc{\theta}$ is chosen so that $\|\vc{\theta} - \vc{\theta}^{(\rs{cp}_{1})}\|$ is sufficiently small, $\theta_{2} < \theta^{(\rs{cp}_{1})}_{2}$ and
\begin{align}
\label{eq:gamma sign 2}
  [\nabla \gamma(\vc{\theta}^{(\rs{cp}_{1})})]_{2} < 0, \qquad \br{\nabla \gamma(\vc{\theta}^{(\rs{cp}_{1})}), \vc{t}^{(\vc{\theta}^{(\rs{cp}_{1})})}_{1}} < 0,
\end{align}
where we recall that $\vc{t}^{(\vc{\theta})}_{1}$ is defined in \lem{geometric sign 1}. The first inequality follows from the convexity of $\Gamma^{\rs{in}}$ and the definition \eq{cp 1} (see \fig{2d-GJN-C3}). For the second inequality, let
\begin{align*}
  \vc{\theta}^{(1,\max)} = \argsup_{\vc{\theta} \in \Gamma^{\rs{in}}} \theta_{1}.
\end{align*}
If $\vc{\theta}^{(\vc{\theta}^{(\rs{cp}_{1})})} = \vc{\theta}^{(1,\max)}$, then the second inequality of \eq{gamma sign 2} is immediate because $\nabla \gamma(\vc{\theta}^{\vc{\theta}^{(1,\max)}})$ is proportional to $\vcn{e}_{1} \equiv (1,0)$ while $\vc{t}^{(\vc{\theta}^{(\rs{cp}_{1})})}_{1} < \vc{0}$ by \lem{geometric sign 1}. Otherwise, assume that $\vc{\theta}^{(\vc{\theta}^{(\rs{cp}_{1})})} \ne \vc{\theta}^{(1,\max)}$, and let $f$ be a function from $\dd{R}$ to $\dd{R}$ such that $\theta_{2} = f(\theta_{1})$ is determined by $\gamma_{s,2}(\vc{\theta}) = 0$. We then observe that $f(\theta_{1})$ is increasing convex in $\theta_{1}$, and its derivative is smaller than that of the curve $\gamma(\vc{\theta}) = 0$ at $\vc{\theta} = \vc{\theta}^{(\rs{cp}_{k})}$ because $\vc{\theta}^{(\rs{cp}_{k})}$ is only one cross point of those two curves for $\theta_{1} > 0$ and $\Gamma^{\rs{in}}$ is not empty. Again the tangent vector $\vc{t}^{(\vc{\theta}^{(\rs{cp}_{1})})}_{1} < \vc{0}$ by \lem{geometric sign 1}, and therefore the second inequality of \eq{gamma sign 2} must hold.
\end{proof}

\begin{lemma}
\label{lem:divergent 1}
Under $\widetilde{\dd{P}}^{(\vc{\theta})}$, all stations of the GJN are weakly unstable if $\nabla \gamma(\vc{\theta}) \ge 0$ for $\vc{\theta} \in \dd{R}^{d}$.
\end{lemma}
\begin{proof}
By \eq{new gradient 1} and the choice of $\vc{\theta}$, $\gamma^{(\vc{\theta})}(\vc{0}) \ge \vc{0}$, and therefore (c) of \lem{unstable 1} proves this lemma.
\end{proof}

\subsection{Proofs of theorems and their corollary}
\label{sect:theorem}

In this subsection, we prove Theorems \thrt{decay 1} and \thrt{decay 2} and \cor{decay 1}.

\begin{proof*}{Proof of \thr{decay 1}}
We apply the procedure in \sectn{method}. (a) Fix $k \in \sr{J}$, and put $B(x) = x \vcn{e}_{k} + B_{0}$ and let $\tau_{x} = \tau^{\rs{in}}_{B(x)}$. Since $B_{0}$ is a compact set, \eq{bounded 1} is satisfied. Hence, all the steps 1)--5) are verified by Lemmas \lemt{Y 1} and \lemt{finite 1}. 

\noindent (b) We first prove \eq{upper boundary 1}. Similar arguments to (a), we put $B(x) = x \vc{c} + B_{0}$ and let $\tau_{x} = \tau^{\rs{in}}_{B(x)}$. We first prove, for each $A \ne \emptyset$, 
\begin{align}
\label{eq:upper 5}
 & \limsup_{x \to \infty} \frac 1x \log \dd{P}(\vc{L} \in x \vc{c} + B_{0}) \le - r_{A}(\vc{c}).
\end{align}
We only need to verify step 3), that is, $\dd{E}_{\nu^{-}_{i}}(f_{\vc{\theta}}(X(0))) < \infty$ for all $i \in A$ because \lem{finite 1} can not be used. We here use the assumption (A1), then it is not hard to see that, for $i \in A$, $\varphi_{i}(\vc{\theta}) < \infty$ implies that $\dd{E}_{\nu_{i}}(f_{\vc{\theta}}(X(0))) < \infty$. The latter finiteness implies that $\dd{E}_{\nu^{-}_{i}}(f_{\vc{\theta}}(X(0))) < \infty$ as shown in the proof of \lem{finite 1}. Thus, Step 3) is verified, and \eq{upper 5} is obtained. Taking the minimum of the right-hand side of \eq{upper 5} for $A \subset \sr{J}$ and $A \ne \emptyset$, we obtain \eq{upper boundary 1}.

We next prove \eq{upper marginal 1}.  Let $\vc{\theta} = u \vc{c}_{0}$ for $u > 0$ for an arbitrarily chosen $\vc{c}_{0} \in \overleftarrow{U}_{d}$, and put
\begin{align*}
  B(x) = \{\vc{z} \in \dd{Z}_{+}^{d}; x < \br{u\vc{c}_{0}, \vc{z}} \le x+1\},
\end{align*}
then \eq{bounded 1} is satisfied, and therefore we can use \lem{finite 1}. By (A1), Step 3 works as shown in (a). For Step 4, we put $\ol{a}_{0}(\vc{\theta}) = 0$ and $\ol{a}_{1}(\vc{\theta}) = u$, then \eq{bound 2} is satisfied. Thus, if we choose $u \vc{c}_{0} \in \Gamma_{A}$, all the steps works, and we have
\begin{align*}
  \limsup_{x \to \infty} \frac 1x \log \dd{P}(x < \br{\vc{c}_{0},\vc{L}} \le x + 1) \le - u,
\end{align*}
as long as $\varphi_{i}(\vc{\theta}) < \infty$ for all $i \in A$. Because $\Gamma_{A}$ is open set, this obviously implies that
\begin{align*}
  \limsup_{x \to \infty} \frac 1x \log \dd{P}(\br{\vc{c}_{0},\vc{L}} > x) \le - u.
\end{align*}
Furthermore, we have, for any $\epsilon \in (0,u]$ and some $x_{0}>0$, 
\begin{align*}
  \dd{P}(\br{\vc{c}_{0},\vc{L}} > x) \le e^{-(u-\epsilon/2)x},
\end{align*}
and therefore, for all $u > 0$ such that $u \vc{c}_{0} \in \Gamma_{A}$, we have $\dd{E}(e^{(u-\epsilon) \br{\vc{c}_{0},\vc{L}}}) < \infty$, which implies that $\dd{E}(e^{\br{\vc{\theta},\vc{L}}}) < \infty$ for all $\vc{\theta} < (u-\epsilon) \vc{c}_{0}$. Since $\overleftarrow{\Gamma}^{\rs{in}}_{A}$ is an open set, this further implies that $\dd{E}(e^{\br{\vc{\theta},\vc{L}}}) < \infty$ for $\vc{\theta} \in \overleftarrow{\Gamma}^{\rs{in}}_{A}$. For a given $\vc{c}$, we choose $u>0$ such that $u \vc{c} \in \overleftarrow{\Gamma}^{\rs{in}}_{A}$, and put $\vc{\theta} = u \vc{c}$. Then,
\begin{align*}
  e^{u x}\dd{P}( \br{\vc{c},\vc{L}} > x) \le  \dd{E}(e^{\br{\vc{\theta},\vc{L}}}) < \infty,
\end{align*}
as long as $\varphi_{i}(\vc{\theta}) < \infty$ for all $i \in A$, and therefore we have
\begin{align*}
  \limsup_{x \to \infty} \frac 1x \log \dd{P}(\br{\vc{c},\vc{L}} > x) \le -m_{A}(\vc{c}).
\end{align*}
Thus, we complete the proof by taking the minimum of the right-hand side of the above inequality over $A \subset \sr{J} \setminus \emptyset$.
\end{proof*}

\begin{proof*}{Proof of \thr{decay 2}}
We apply the lower bound procedure 1')--5'). Because of symmetry, it suffices to prove for $k=1$. (a) Put  $A = \{1\}$, and let $B(x) = x \vcn{e}_{1}+B_{0}$ and $\tau_{x} = \tau^{\rs{in}}_{B(x)}$. We choose $\vc{\theta}$ such that $\gamma(\vc{\theta}) > 0$ and $\gamma_{s,2}(\vc{\theta}) < 0$, then Step 1') works. Step 2') is obviously verified because $Y(\tau_{x})$ does not decrease as $x$ gets large. Step 3') is also obvious because $B_{0}$ is compact. For step 4'), we can take any bounded set for $U$. Then, if we take $\vc{\theta}$ which is sufficiently close to $\vc{\theta}^{(\rs{cp}_{k})}$, then \eq{hitting 1} holds by \lem{hitting 1}, while \eq{U 1} obviously holds. This completes the procedure, and \eq{lower 1a}  is obtained.

\noindent (b) We restrict the initial state in a bounded set $C$ such that $C \subset \partial S_{1} \times \dd{R}_{+}^{2d}$ and
\begin{align*}
  \dd{E}(f(X(0))1(X(0) \in C)) > 0.
\end{align*}
Let $\vc{c} \in {\rm Corn}(\overleftarrow{\Gamma}^{\rs{in}} \cap \partial \Gamma^{\rs{in}})$, which implies that $\nabla g(u\vc{c}) \ge \vc{0}$ for $u\vc{c} \in {\rm Corn}(\overleftarrow{\Gamma}^{\rs{in}} \cap \partial \Gamma^{\rs{in}})$ by the convexity of $\Gamma^{\rs{in}}$. Choose $z_{0}$ such that $\vc{x} \equiv (\vc{z}, \vc{y}) \in C$ implies that $\max_{i \in \sr{J}} z_{i} < z_{0}$. We let
\begin{align*}
  B(x) = \{\vc{z} \in \dd{Z}_{+}^{d}; \br{\vc{c}, \vc{z}} > x\}, \qquad x > z_{0},
\end{align*}
and let $\tau_{x} = \tau^{\rs{in}}_{B(x)}$. Then, Step 2') is obviously valid. Because the initial state is in $C$,
\begin{align*}
  u x < \br{u\vc{c}, \vc{L}(\tau_{x})} \le u x+1.
\end{align*}
Hence, the condition \eq{bound 3} in Steps 3') is satisfied for $x \ge z_{0}$. Furthermore, if we take $\vc{\theta} = u \vc{c}$ for the change of measure, then all the stations are weakly stable by \lem{divergent 1}, which implies that
\begin{align}
\label{eq:tau finite 1}
  \dd{E}(f(X(0))1(X(0) \in C, \tau_{x} < \infty)) = \dd{E}(f(X(0))1(X(0) \in C)) > 0,
\end{align}
and therefore \eq{hitting 1} is satisfied for $A = \sr{J}$. Thus, all the steps work well, and the proof is completed.

\noindent (c) We take the same $B(x), \tau_{x}$ and $C$ as in (a). Let ${\rm Corn}(\vc{a}, \vc{b}) = \{\vc{x} \in \dd{R}_{+}^{2}; s \vc{a} + t \vc{b}, s,t \ge 0\}$ for $\vc{a}, \vc{b} \in \dd{R}_{+}^{2}$. For $\vc{c} \in {\rm Corn}(\overleftarrow{\Gamma}^{\rs{in}} \cap \partial \Gamma^{\rs{in}}_{1})$, we separately consider the two cases that $\vc{c} \in {\rm Corn}(\vcn{e}_{1}, \vc{\theta}^{(\rs{cp}_{1})})$ or not. If $\vc{c} \not\in {\rm Corn}(\vcn{e}_{1}, \vc{\theta}^{(\rs{cp}_{1})})$, the asymptotic is covered by \eq{lower 2}. Thus, we assume that $\vc{c} \in {\rm Corn}(\vcn{e}_{1}, \vc{\theta}^{(\rs{cp}_{1})})$. We first choose $u > 0$ and $\vc{c} \in \overrightarrow{U}_{2}$ such that $u \vc{c} = \vc{\theta}^{(\rs{cp}_{1})}$, and make the change of measure for $\vc{\theta} = \vc{\theta}^{(\rs{cp}_{1})}$. Then, we have \eq{tau finite 1} by \lem{hitting 1}. Hence, we have \eq{lower 1b}. We next let $u = \theta^{({\rm \rs{cp}_{1}})}_{1}$ and let $\vc{c} = \vcn{e}_{1}$. In this case, we also have \eq{lower 1b} by \eq{lower 1a}. We finally consider the case that $u \vc{c} = s \vcn{e}_{1} + t \vc{\theta}^{(\rs{cp}_{1})} \in \overleftarrow{\Gamma}^{\rs{in}} \cap \partial \Gamma^{\rs{in}}_{1}$. Let $u = \theta^{({\rm \rs{cp}_{1}})}_{1}/c_{1}$, $s = \theta^{(\rs{cp}_{1})}_{1}$ and $t = \theta^{(\rs{cp}_{1})}_{1} c_{2}/\theta^{(\rs{cp}_{1})}_{2}$, then $u \vc{c} = (\theta^{(\rs{cp}_{1})}_{1}, \theta^{(\rs{cp}_{1})}_{1} c_{2}/c_{1})$ is on $\overleftarrow{\Gamma}^{\rs{in}} \cap \partial \Gamma^{\rs{in}}_{1}$. Hence, we have \eq{lower 1b}.
\end{proof*}

\begin{proof*}{Proof of \cor{decay 1}}
(a) For $d=2$, from Theorems \thrt{decay 1} and \thrt{decay 2}, we have
\begin{align}
\label{eq:decay 3}
 & \lim_{x \to \infty} \frac 1x \log \dd{P}(\vc{L} \in x \vcn{e}_{k} + B_{0}) = - r_{*}(\vcn{e}_{k}), \qquad k = 1,2.
\end{align}
Then, we can apply the same algorithm as in Theorem 4.1 of \cite{Miya2009} to find $r_{*}(\vcn{e}_{k})$, which shows that \eq{cp 1a} and \eq{cp 1b} have a unique solution $\vc{\delta} \equiv (\delta_{1}, \delta_{2})$, and $r_{*}(\vcn{e}_{k}) = \delta_{k}$. This proves \eq{decay 1}.

(b) \eq{upper 3} is immediate from (b) of \thr{decay 1} for $A = \sr{J}$. It remains to prove \eq{decay 2}. We first consider the marginal distributions in the coordinate directions. By \eq{upper marginal 1} of \thr{decay 1} for $d=2$, it follows from $ \varphi_{1}(0) \le 1$ that
\begin{align*}
   & \limsup_{x \to \infty} \frac 1x \log \dd{P}(\br{\vcn{e}_{1},\vc{L}} >  x) \le - m_{\{1\}}(\vcn{e}_{1}) = - \sup\{u; u \vcn{e}_{1} \in \overleftarrow{\Gamma}^{\rs{in}}_{1}\}.
\end{align*}
This combining with \eq{lower 1b} concludes that
\begin{align*}
   & \limsup_{x \to \infty} \frac 1x \log \dd{P}(\br{\vcn{e}_{1},\vc{L}} >  x) = - \sup\{u; u \vcn{e}_{1} \in \overleftarrow{\Gamma}^{\rs{in}}_{1}\} = - \delta_{1},
\end{align*}
and therefore $\varphi(\theta_{1},0)$ is finite for $\theta_{1} < \delta_{1}$ and diverges for $\theta_{1} > \delta_{1}$. Similarly, $\varphi(0,\theta_{2})$ is finite for $\theta_{2} < \delta_{2}$. Since $\varphi_{2}(\theta_{1}) \le \varphi(\theta_{1},0)$ and $\varphi_{1}(\theta_{2}) \le \varphi(\theta_{2},0)$, it follows again from \eq{upper marginal 1} that
\begin{align*}
   \limsup_{x \to \infty} \frac 1x \log \dd{P}(\br{\vc{c},\vc{L}} >  x) & \le - m_{\{1,2\}}(\vc{c}) = - \sup\{u; u \vc{c} \in \overleftarrow{\Gamma}^{\rs{in}}, \varphi_{1}(\theta_{2}), \varphi_{2}(\theta_{1}) < \infty\}\\
   & \le - \sup\{u; u \vc{c} \in \overleftarrow{\Gamma}^{\rs{in}}, \theta_{i} < \delta_{i}, \i=1,2 \} = - \sup\{u; u \vc{c} \in \sr{D}_{2}\}.
\end{align*}
Thus, we got the upper bound. By \eq{lower 2} and \eq{lower 1b} of \thr{decay 2}, this upper bound becomes a lower bound. Hence, we have \eq{decay 2}.
\end{proof*}


\begin{thebibliography}{23}
\expandafter\ifx\csname natexlab\endcsname\relax\def\natexlab#1{#1}\fi
\expandafter\ifx\csname url\endcsname\relax
  \def\url#1{\texttt{#1}}\fi
\expandafter\ifx\csname urlprefix\endcsname\relax\def\urlprefix{URL }\fi
\providecommand{\eprint}[2][]{\url{#2}}

\bibitem[{Asmussen(2003)}]{Asmu2003}
\textsc{Asmussen, S.} (2003).
\newblock \textit{Applied probability and queues}, vol.~51 of
  \textit{Applications of Mathematics (New York)}.
\newblock 2nd ed. Springer-Verlag, New York.
\newblock Stochastic Modelling and Applied Probability.

\bibitem[{Avram et~al.(2001)Avram, Dai and Hasenbein}]{AvraDaiHase2001}
\textsc{Avram, F.}, \textsc{Dai, J.~G.} and \textsc{Hasenbein, J.~J.} (2001).
\newblock Explicit solutions for variational problems in the quadrant.
\newblock \textit{Queueing Systems}, \textbf{37} 259--289.

\bibitem[{Baccelli and Br{\'e}maud(2003)}]{BaccBrem2003}
\textsc{Baccelli, F.} and \textsc{Br{\'e}maud, P.} (2003).
\newblock \textit{Elements of queueing theory: Palm martingale calculus and
  stochastic recurrences}, vol.~26 of \textit{Applications of Mathematics}.
\newblock 2nd ed. Springer, Berlin.

\bibitem[{Braverman et~al.(2015)Braverman, Dai and Miyazawa}]{BravDaiMiya2015}
\textsc{Braverman, A.}, \textsc{Dai, J.} and \textsc{Miyazawa, M.} (2015).
\newblock Heavy traffic approximation for the stationary distribution of a
  generalized jackson network: the {BAR} approach.
\newblock Submitted for publication.

\bibitem[{Chen and Mandelbaum(1991{\natexlab{a}})}]{ChenMand1991}
\textsc{Chen, H.} and \textsc{Mandelbaum, A.} (1991{\natexlab{a}}).
\newblock Discrete flow networks: bottlenecks analysis and {fluid}
  approximations.
\newblock \textit{Mathematics of Operations Research}, \textbf{16} 408--446.

\bibitem[{Chen and Mandelbaum(1991{\natexlab{b}})}]{ChenMand1991b}
\textsc{Chen, H.} and \textsc{Mandelbaum, A.} (1991{\natexlab{b}}).
\newblock Stochastic discrete flow networks: Diffusion approximation and
  bottlenecks.
\newblock \textit{Annals of Probability}, \textbf{19} 1463--1519.

\bibitem[{Dai and Miyazawa(2011)}]{DaiMiya2011}
\textsc{Dai, J.~G.} and \textsc{Miyazawa, M.} (2011).
\newblock Reflecting {Brownian} motion in two dimensions: Exact asymptotics for
  the stationary distribution.
\newblock \textit{Stochastic Systems}, \textbf{1} 146--208.
\newblock \urlprefix\url{http://dx.doi.org/10.1214/10-SSY022}.

\bibitem[{Dai and Miyazawa(2013)}]{DaiMiya2013}
\textsc{Dai, J.~G.} and \textsc{Miyazawa, M.} (2013).
\newblock Stationary distribution of a two-dimensional {SRBM}: geometric views
  and boundary measures.
\newblock \textit{Queueing Systems}, \textbf{74} 181--217.
\newblock \urlprefix\url{http://dx.doi.org/10.1007/s11134-012-9339-1}.

\bibitem[{Davis(1984)}]{Davi1984}
\textsc{Davis, M. H.~A.} (1984).
\newblock Piecewise deterministic {M}arkov processes: a general class of
  non-diffusion stochastic models.
\newblock \textit{Journal of Royal Statist.\ Soc.\, series B}, \textbf{46}
  353--388.

\bibitem[{Glynn and Whitt(1994)}]{GlynWhit1994a}
\textsc{Glynn, P.~W.} and \textsc{Whitt, W.} (1994).
\newblock Large deviations behavior of counting processes and their inverses.
\newblock \textit{Queueing Systems}, \textbf{17} 107--128.

\bibitem[{Jacod and Shiryaev(2003)}]{JacoShir2003}
\textsc{Jacod, J.} and \textsc{Shiryaev, A.~N.} (2003).
\newblock \textit{Limit Theorems for stochastic processes}.
\newblock 2nd ed. Springer, Berlin.

\bibitem[{Kingman(1961)}]{King1961b}
\textsc{Kingman, J. F.~C.} (1961).
\newblock A convexity property of positive matrix.
\newblock \textit{The Quarterly Journal of Mathematics}, \textbf{12} 283--284.

\bibitem[{Kunita and Watanabe(1963)}]{KuniWata1963}
\textsc{Kunita, H.} and \textsc{Watanabe, T.} (1963).
\newblock Notes on transformations of markov processes connected with
  multiplicative functionals.
\newblock \textit{Memoirs of the Faculty of Science, Kyushu}, \textbf{17}
  181--191.

\bibitem[{Majewski(2009)}]{Maje2009}
\textsc{Majewski, K.} (2009).
\newblock Functional continuity and large deviations for the behavior of
  single-class queueing networks.
\newblock \textit{Queueing Systems: Theory and Applications}, \textbf{61}
  203--241.

\bibitem[{Miyazawa(2009)}]{Miya2009}
\textsc{Miyazawa, M.} (2009).
\newblock Tail decay rates in double {QBD} processes and related reflected
  random walks.
\newblock \textit{Math. Oper. Res.}, \textbf{34} 547--575.
\newblock \urlprefix\url{http://dx.doi.org/10.1287/moor.1090.0375}.

\bibitem[{Miyazawa(2011)}]{Miya2011}
\textsc{Miyazawa, M.} (2011).
\newblock Light tail asymptotics in multidimensional reflecting processes for
  queueing networks.
\newblock \textit{TOP}, \textbf{19} 233--299.

\bibitem[{Miyazawa(2014)}]{Miya2014}
\textsc{Miyazawa, M.} (2014).
\newblock Superharmonic vector for a nonnegative matrix with {QBD} block
  structure and its application to a {Markov} modulated two dimensional
  reflecting process.
\newblock Working paper.

\bibitem[{Miyazawa(2015)}]{Miya2015a}
\textsc{Miyazawa, M.} (2015).
\newblock A superharmonic vector for a nonnegative matrix with {QBD} block
  structure and its application to a {Markov} modulated two dimensional
  reflecting process.
\newblock \textit{Queueing Systems}, \textbf{81} 1--48.

\bibitem[{Miyazawa(2017)}]{Miya2017}
\textsc{Miyazawa, M.} (2017).
\newblock A unified approach for large queue asymptotics in a heterogeneous
  multiserver queue.
\newblock \textit{Advances in Applied Probability} 49, 182-220.
\newblock Supplemented version, arxiv.org/abs/1510.01034.

\bibitem[{Palmowski and Rolski(2002)}]{PalmRols2002}
\textsc{Palmowski, Z.} and \textsc{Rolski, T.} (2002).
\newblock A technique of the exponential change of measure for {Markov}
  processes.
\newblock \textit{Bernoulli}, \textbf{8} 767--785.

\bibitem[{Reiman(1984)}]{Reim1984}
\textsc{Reiman, M.~I.} (1984).
\newblock Open queueing networks in heavy traffic.
\newblock \textit{Mathematics of Operations Research}, \textbf{9} 441--458.

\bibitem[{Sadowsky and Szpankowski(1995)}]{SadoSzpa1995}
\textsc{Sadowsky, J.} and \textsc{Szpankowski, W.} (1995).
\newblock The probability of large queue lengths and waiting times in a
  heterogeneous multiserver queue {I}: {Tight} limits.
\newblock \textit{Advances in Applied Probability}, \textbf{27} 532--566.

\bibitem[{Whitt(2002)}]{Whit2002}
\textsc{Whitt, W.} (2002).
\newblock \textit{Stochastic-process limits}.
\newblock Springer, New York.

\end{thebibliography}

\def\cprime{$'$} \def\cprime{$'$} \def\cprime{$'$} \def\cprime{$'$}
  \def\cprime{$'$} \def\cprime{$'$} \def\cprime{$'$}

\end{document}